\documentclass{amsart}
\usepackage{amsmath,amsthm,amsfonts,amssymb,amscd}
\usepackage{framed}
\usepackage{nicefrac}
\usepackage{color,graphicx}
\usepackage[x11names, rgb]{xcolor}
\usepackage{tikz}
\usetikzlibrary{cd}
\usepackage{hyperref}
\usepackage{pdfsync}
\usepackage{enumitem}
\usepackage[capitalize,nameinlink]{cleveref}
\Crefname{conjecture}{Conjecture}{Conjectures}
\usepackage[utf8]{inputenc}
\usepackage{tableau,ytableau}
\usepackage[position=bottom]{subfig}
\usepackage{youngtab}

\newtheorem{thm}{Theorem}[section]
\newtheorem{cor}[thm]{Corollary}
\newtheorem{prop}[thm]{Proposition}
\newtheorem{lemma}[thm]{Lemma}
\newtheorem{example}[thm]{Example}

\newtheorem{remark}[thm]{Remark}

\newcommand{\C}{{\mathbb C}}

\newcommand{\bP}{{\mathbb{P}}}

\newcommand{\Z}{{\mathbb Z}}

\newcommand{\cS}{{\mathcal S}}
\newcommand{\fS}{{\mathfrak S}}
\newcommand{\cQ}{{\mathcal Q}}

\newcommand{\cO}{{\mathcal O}}

\newcommand{\Fl}{\mathrm{Fl}}
\newcommand{\Gr}{\mathrm{Gr}}

\DeclareMathOperator{\QH}{QH}
\DeclareMathOperator{\QK}{QK}
\DeclareMathOperator{\K}{K}
\DeclareMathOperator{\Rep}{Rep}
\DeclareMathOperator{\ch}{ch}

\begin{document}
\title[QK theory via Schur bundles]{Quantum K theory of Grassmannians, Wilson line operators, and Schur bundles}

\author{Wei Gu}
\address{Department of Physics MC 0435, 850 West Campus Drive,
Virginia Tech University, Blacksburg VA  24061
USA}
\email{weig8@vt.edu}

\author{Leonardo C.~Mihalcea}
\address{
Department of Mathematics, 
225 Stanger Street, McBryde Hall,
Virginia Tech University, 
Blacksburg, VA 24061
USA
}
\email{lmihalce@vt.edu}

\author{Eric Sharpe}
\address{Department of Physics MC 0435, 850 West Campus Drive,
Virginia Tech University, Blacksburg VA  24061
USA}
\email{ersharpe@vt.edu}

\author{Hao Zou}
\address{Beijing Institute of Mathematical Sciences and Applications, Beijing 101408, China}
\address{Yau Mathematical Sciences Center, Tsinghua University, Beijing 100084, China}
\email{hzou@vt.edu}


\begin{abstract} We prove a `Whitney' presentation, and a `Coulomb branch' presentation, for 
the torus equivariant quantum K theory of the Grassmann manifold $\Gr(k;n)$, inspired from 
physics, and stated in an earlier paper. 
The first presentation is obtained by quantum deforming 
the product of the Hirzebruch $\lambda_y$
classes of the tautological bundles. In physics, 
the $\lambda_y$ classes arise as certain Wilson line operators. The second presentation is obtained 
from the Coulomb branch equations involving the partial derivatives
of a twisted superpotential from supersymmetric gauge theory. This is closest to a presentation 
obtained by Gorbounov and Korff, utilizing integrable systems techniques. Algebraically, we relate
the Coulomb and Whitney presentations utilizing transition matrices from 
the (equivariant) Grothendieck polynomials 
to the (equivariant) complete homogeneous symmetric polynomials. Along the way, 
we calculate 
K-theoretic Gromov-Witten invariants of wedge powers of the tautological bundles on $\Gr(k;n)$, 
using the `quantum=classical' statement. \end{abstract}

\maketitle

\setcounter{tocdepth}{1}
\tableofcontents

\section{Introduction} Based on predictions inspired by physics, in \cite{Gu:2020zpg}
we conjectured presentations by generators and relations for the quantum K theory ring of 
the Grassmannian and for the Lagrangian Grassmannian.
The main goal in this paper is to provide rigorous mathematical proofs of these 
statements in the Grassmannian case, in the more general equivariant context.

Let $\Gr(k;n)$ denote the Grassmann manifold parametrizing linear subspaces of dimension $k$ in $\C^n$, and 
let $0 \to \cS \to \C^n \to \cQ \to 0$ be the tautological sequence, where $\mathrm{rk}(\cS) = k$. 
An influential result by Witten \cite{witten:verlinde} states that 
$(\QH^*(\Gr(k;n)),\star)$, the quantum cohomology ring  
of the Grassmannian, is determined by the `quantum Whitney relations':
\begin{equation}\label{E:wittenrel} c(\cS) \star c (\cQ) = c(\C^n) + (-1)^k q \/, \end{equation}
where $c(E)= 1+ c_1(E)+ \ldots + c_e(E)$ is the total Chern class of the rank $e$ bundle $E$. 
This equation leads to a presentation of
$\QH^*(\Gr(k;n))$ by generators and relations:
\begin{equation}\label{E:introQHpres} 
\QH^*(\Gr(k;n))= \frac{\Z[q][e_1(x), \ldots, e_k(x); e_1(\tilde{x}), \ldots, e_{n-k}(\tilde{x})]}{\left\langle \left(
\sum_{i=0}^k e_i(x) \right) \left(\sum_{j=0}^{n-k} e_j(\tilde{x}) \right) = 1 + (-1)^k q \right\rangle } \/. 
\end{equation}
Here $e_i(x) = e_i(x_1, \ldots , x_k), e_j(\tilde{x}) = e_j(\tilde{x}_1, \ldots , \tilde{x}_{n-k})$ 
denote the elementary symmetric polynomials, and $x_i, \tilde{x}_j$ correspond to 
the Chern roots of $\cS$, respectively $\cQ$.

Let $T$ be the torus of invertible diagonal matrices with its usual action
on $\Gr(k;n)$.~In this paper we generalize Witten's relations 
\eqref{E:wittenrel} from the quantum cohomology to the  
$T$-equivariant quantum $\K$ ring of $\Gr(k;n)$, 
defined by Givental and Lee
\cite{givental:onwdvv,givental.lee:quantum,lee:QK}.
Denote this ring by $\QK_T(\Gr(k;n))$ and by $\K_T(\Gr(k;n))$
the ordinary equivariant $\K$-theory ring.

For a vector bundle $E$, denote by 
$\lambda_y(E) = 1 + y E + \ldots +  y^e \wedge^e E$
the Hirzebruch $\lambda_y$-class in the $\K$ theory ring.
The $\K$-theoretic Whitney relations are 
\[ \lambda_y(\cS) \cdot \lambda_y(\cQ) = \lambda_y(\C^n) \quad \in \K_T(\Gr(k;n)) \/. \]
The main geometric result in this paper is a quantum deformation of these relations. 
\begin{thm}\label{thm:intro} The following equality holds in $\QK_T(\Gr(k;n))$: 
\begin{equation}\label{E:introQKWhitney}\lambda_y(\cS) \star \lambda_y(\cQ)  =  
\lambda_y(\C^n) - \frac{q}{1-q} y^{n-k} (\lambda_y(\cS) -1) \star \det \cQ \/.
\end{equation}
\end{thm}
In the theorem we regard $\C^n$ as the $T$-module with weight space decomposition 
$\C^n=\C_{t_1} \oplus \ldots \oplus \C_{t_n}$. 
If we set $T_1:=e^{t_1} , \ldots , T_n:=e^{t_n}$, 
then $\K_T(pt)$ is the Laurent polynomial ring $\Z[T_1^{\pm 1}, \ldots, T_n^{\pm 1}]$, 
and $\lambda_y(\C^n) = \prod_{i=1}^n (1+y T_i)$, see~\S \ref{sec:preliminaries}.

The equality \eqref{E:introQKWhitney} may be translated into an abstract `quantum K-theoretic 
Whitney' presentation of the ring $\QK_T(\Gr(k;n))$ by generators and relations; 
this is stated in \Cref{thm:intro2} and \Cref{thm:qkpres} below. 
In this case one uses variables $X_i, \tilde{X}_j$, interpreted as 
exponentials of the Chern roots: $X_i = e^{x_i}$, $\tilde{X}_j= e^{\tilde{x}_j}$.
This naturally generalizes the classical presentations of $\K(\Gr(k;n))$ by Lascoux 
\cite{lascoux:anneau}, and of $\QH^*(\Gr(k;n))$ by Witten from \eqref{E:introQHpres} above.

In \cite{Gu:2020zpg} we conjectured a second `Coulomb branch presentation' of the ring
$\QK_T(\Gr(k;n))$, 
and in this paper we prove an equivariant generalization of it.
Let $W$ be a twisted
superpotential \cite{morrison.plesser,AHKT:mirror,Closset:2016arn,Gu:2020zpg}
arising in the study of supersymmetric gauge theory; 
cf.~\S \ref{sect:physbase} below.
The expression for $W$ depends
on the exponentials $X_1, \ldots, X_k$. 
In this context, the exterior powers $\wedge^i \cS, \wedge^j \cQ$ 
arise as certain Wilson line operators considered in the physics literature
\cite{Jockers:2018sfl,Jockers:2019lwe,Jockers:2019wjh,Ueda:2019qhg}. 
The Coulomb branch (or vacuum) equations
for $W$ are
\begin{equation}\label{E:introCoulomb} \exp\left(\frac{\partial W}{\partial \ln X_i}\right) = 1 \/, \quad 1 \le i \le k \/.\end{equation}
It is convenient to work
with the `shifted Wilson line operators', or, equivalently, with the (dual) $\K$-theoretic 
Chern roots
\begin{equation}\label{E:cov} z_i = 1 - X_i \/, \quad (1 \le i \le k) \/; \quad \zeta_j = 1 - T_j\/,  \quad (1 \le j \le n)\/. \end{equation}
The Coulomb branch equations show that $z_i$ are the roots 
of a `characteristic polynomial':   
\begin{equation} \label{eq:introchar}
f(\xi,z,\zeta,q) \: = \: 
\xi^n \: + \: \sum_{i=0}^{n-1} (-1)^{n-i} \xi^i \hat{g}_{n-i}(z,\zeta,q) \/,
\end{equation}
where the polynomials $\hat{g}_j(z,\zeta,q)$, defined in \eqref{E:gell} below,
 are symmetric both in $z_i$'s 
and $\zeta_j$'s.~Our main result in this paper is that
the relations from \Cref{thm:intro} generate the ideal
of relations of $\QK_T(\Gr(k;n))$, and, furthermore,
that the Vieta relations satisfied by the roots of the polynomial \eqref{eq:introchar} 
are algebraically equivalent to those from \Cref{thm:intro}; 
see \Cref{thm:main-result}
below. Specifically,
denote by $\tilde{z}_j = 1- \tilde{X}_j$ the K-theoretic Chern roots of the dual of $\cQ$. 
\begin{thm}\label{thm:intro2}
The following two rings give presentations by generators and relations of $\QK_T(\Gr(k;n))$:
\begin{enumerate}[label=(\alph*)]
\item (Coulomb branch presentation)
The ring 
$\widehat{\QK}_T(\Gr(k;n))$ given by
\begin{equation*}\label{E:intro-Coulomb-pres} 
{\mathrm K}_T(pt)[[q]] [e_1(z), \cdots, e_k(z), e_1(\hat{z}), \cdots, e_{n-k}(\hat{z})]
\, / \,
\langle \sum_{i+j=\ell} e_i(z) e_j(\hat{z}) - \hat{g}_{\ell}(z,\zeta,q) 
\rangle_{1 \leq \ell \leq n} \/.
\end{equation*}
{Here $\hat{z}=(\hat{z}_1, \ldots, \hat{z}_{n-k})$ and the polynomials $\hat{g}_\ell(z,\zeta,q)$ are defined 
in \eqref{E:gell}.}
\item (QK-theoretic Whitney presentation) The ring 
$\widetilde{\QK}_T(\Gr(k;n))$ given by
\begin{equation*}\label{E:intro:QWhitney-pres}
{\mathrm K}_T(pt)[[q]] [e_1(z),\cdots,e_k(z),e_1(\tilde{z}),\cdots,e_{n-k}(\tilde{z}
)]
\, / \,
\langle \sum_{i+j=\ell} e_{i}(z) e_{j}(\tilde{z}) \: - \: \tilde{g}_{\ell}
(z,\zeta,q)
\rangle_{1 \leq \ell \leq n} \/, 
\end{equation*}
and where the polynomials $\tilde{g}_\ell$ are defined in \eqref{E:tildeg2}.
\end{enumerate}
In each situation, $e_i(z)$ is sent to
$\sum_{p =0}^i (-1)^p {k -p \choose i-p} \wedge^p \cS$, for $1 \le i \le k$.
\end{thm}
The relations giving $\widetilde{\QK}_T(\Gr(k;n))$ are obtained from those in \Cref{thm:intro}
by the change of variables \eqref{E:cov} and $\tilde{X}_j = 1-\tilde{z}_j$. 

The two presentations satisfy some remarkable algebraic properties. In both situations,
the polynomials $\hat{g}_\ell$ and $\tilde{g}_\ell$ do not depend on $q$ if $1 \le \ell \le n-k$;
in other words, the first $n-k$ relations are `classical'. Furthermore, we may use the first 
$n-k$ relations to eliminate the variables
$\hat{z}_j,\tilde{z}_j$ for $1 \le j \le n-k$. We prove that in $\widehat{\QK}_T(\Gr(k;n))$,
\[ (-1)^\ell e_\ell(\hat{z}) = G_\ell'(z,\zeta):= \sum_{i+j = \ell} (-1)^j G_i(z) e_j(\zeta) \/, \]
where $G_j(z)$ is the Grothendieck polynomial; see \Cref{prop:Ehat_groth}. 
Note that the elimination of variables process
led us to define $G_\ell'(z,\zeta)$, which is an equivariant deformation of $G_\ell(z)$.
This is different from another equivariant deformation of $G_j(z)$, 
the factorial Grothendieck polynomial defined by McNamara 
\cite{mcnamara:factorial}, and which represents equivariant 
Schubert classes. (The latter are not symmetric in $\zeta$'s.) 
In analogy to the presentation of 
$H^*_T(\Gr(k;n))$, the Cauchy formulae calculating 
$\sum_{i+j=\ell} G_i'(z,\zeta) e_j(\zeta)$
allow us to formulate the relations in $\widehat{\QK}_T(\Gr(k;n))$ as 
expressions involving
$G'_{n-k+1}(z,\zeta), \ldots, G_n'(z,\zeta)$; see \Cref{lemma:twopres} below.  
We have not seen 
a presentation formulated naturally in terms of these polynomials,
even in the ordinary non-quantum ring $\K_T(\Gr(k;n))$.

A similar description holds for the $\QK$-Whitney presentation
$\widetilde{\QK}_T(\Gr(k;n))$, with the caveat that the 
Grothendieck polynomials $G_j(z)$, $G'_\ell(z,\zeta)$ are replaced 
by the complete homogeneous symmetric polynomials $h_j(z)$
and their equivariant versions $h_\ell'(z,\zeta) = \sum_{i+j=\ell} (-1)^j h_i(z) e_j(\zeta)$.
Geometrically, the polynomials $h_\ell'(z,\zeta)$
represent (equivariant)
$\K$-theoretic Chern classes of the tautological quotient bundle
$\cQ$. 
Analyzing the precise algebraic relationship between $G'_j(z,\zeta)$
and $h_j'(z,\zeta)$ lies at the foundation of proving the isomorphism
between the two presentations in \Cref{thm:intro2}; this is done
in \S \ref{sec:pres-iso}. 

One advantage of working with $\K$-theoretic Chern roots 
is that the leading terms of these presentations give presentations of the 
equivariant quantum cohomology ring $\QH^*_T(\Gr(k;n))$; both specialize to
Witten's presentation \eqref{E:introQHpres}. In physics terminology, 
one recovers the $2d$ limit of the theory. This is illustrated in \S \ref{sec:2dlimit}.

As expected from physics \cite{nekrasov.shatashvili}, the Coulomb branch 
equations \eqref{E:introCoulomb}
coincide with the Bethe Ansatz equations considered by Gorbounov and Korff 
\cite{Gorbounov:2014}. In their paper, Gorbounov and Korff 
utilized integrable systems and equivariant localization 
techniques to obtain another presentation of the 
equivariant quantum K ring $\QK_T(\Gr(k;n))$,
different from those in \Cref{thm:intro2}. In order to identify 
the geometric ring $\QK_T(\Gr(k;n))$ with that
given from integrable systems, they require the equivariant quantum K 
Chevalley formula (a formula   
to multiply by a Schubert divisor).~As proved in \cite{BCMP:qkchev}, this formula determines 
the ring structure.

{Aside from a finite generation statement of the `classical' Coulomb branch 
presentation (i.e., modulo $q$), used 
in the proof of \Cref{thm:isomorphism}}, our proofs are logically 
independent from results in \cite{Gorbounov:2014}.
We calculate the product $\lambda_y(\cS) \star \lambda_y(\cQ)$ from \Cref{thm:intro}
directly, based on the 
`quantum=classical' statement by Buch and one of the authors 
\cite{buch.m:qk}, applied to the Schur bundles. This requires the 
calculation of certain push-forwards of the $\K$-theory classes of 
the exterior powers 
$\wedge^i \cS, \wedge^j \cQ$. For this, we rely on sheaf cohomology vanishing 
statements obtained by Kapranov \cite{kapranov:Gr}.
We obtain rather explicit calculations of KGW invariants and quantum K products 
involving Schur bundles, which 
may be interesting in their own right. To illustrate, we prove the following 
quantum deformation of the classical equality 
$\wedge^i \cS \cdot \det(\cQ) = \wedge^{k-i} \cS^* \cdot \det(\C^n) \in \K_T(\Gr(k;n))$
(cf.~\Cref{thm:qdual}, see also \Cref{conj2:qSvee}).
\begin{thm}\label{thm:qdual-intro}The following holds in $\QK_T(\Gr(k;n))$:
\[ (\lambda_y(\cS)-1) \star \det (\cQ) = (1-q) ((\lambda_y(\cS) -1) \cdot \det (\cQ)) \/. \]
Equivalently, for any $i>0$, 
\[ \wedge^i (\cS) \star \det (\cQ) = (1-q) \wedge^{k-i}(\cS^*) \cdot \det(\C^n) \/. \]
\end{thm}

Of course, the $\QK$ Whitney, the Coulomb branch, and the 
presentation from \cite{Gorbounov:2014} are isomorphic.~In 
\S \ref{sec:GKpres} we provide a direct isomorphism from
the non-equivariant Coulomb branch presentation
$\widehat{\QK}(\Gr(k;n))$ to the one from \cite{Gorbounov:2014}.
\begin{footnote}{In an upcoming paper, it will 
be shown that a certain `functional relation', which determines the relations from 
\cite{Gorbounov:2014}, is equivalent to the quantum K Whitney 
relations \eqref{E:introQKWhitney}.}\end{footnote}
The final step in proving \Cref{thm:intro2} is 
to show that there are no other relations 
in the quantum $\K$ ring beyond those stated in \Cref{thm:intro}. 
In the quantum cohomology case,
a result going back to Siebert and Tian
\cite{siebert.tian:on} (see also \cite{fulton.pandh:notes}) states that
if one has a set of quantum relations such that the $q=0$ 
specialization gives the ideal of the `classical' relations, then these generate
the ideal of quantum relations. The proof utilizes the graded Nakayama lemma.
The quantum K theory ring is not graded.
Still, a similar statement holds, utilizing the ordinary Nakayama lemma
for the quantum K ring regarded as a module over the power series ring
$\QK_T(pt):= \K_T(pt)[[q]]$, {together with a finite generation
statement for completed rings; cf. ~\cite[Ex.~7.4,~p.~203]{eisenbud:CAbook} 
or \Cref{rmk:Eis-fg}.}
Such arguments are useful in studying presentations of 
more general quantum K rings, therefore 
we took the opportunity to gather in \Cref{sec:filtered} the relevant results about completions
and filtered modules. ~A second 
appendix contains a worked out example for $\QK_T(\Gr(2;5))$. 

We mention that our initial paper \cite{Gu:2020zpg} also gives physics inspired predictions for 
the quantum K theory of the Lagrangian Grassmannian $\mathrm{LG}(n;2n)$; a mathematical
followup analyzing this situation will be considered elsewhere.

{\em Acknowledgments.} We would like to thank Dave Anderson, Vassily Gorbounov, Takeshi Ikeda,
Christian Korff, Yaoxiong Wen, and Ming Zhang for useful discussions.~Special 
thanks are due to Prof.~Satoshi Naito, whose remarks helped us realize that the hypothesis that $R$ is
complete, or is localized, needs to be added to an earlier ar$\chi$iv version of \Cref{cor:gen}.
W.G. was partially supported by NSF grant PHY-1720321.
This material is partly based upon work supported by the National Science Foundation 
under Grant No.~DMS-1929284 while LM was in residence at the Institute for Computational and 
Experimental Research in Mathematics in Providence, RI, during the 
Combinatorial Algebraic Geometry program. LM was also supported in part by the NSF grant 
DMS-2152294 and a Simons Collaboration Grant. 
E.S. was partially supported by NSF grant
PHY-2014086.

Throughout this project we utilized the Maple based program
\texttt{Equivariant Schubert Calculator}, written by Anders Buch.\begin{footnote}
{The program is available at \texttt{https://sites.math.rutgers.edu/$\sim$asbuch/equivcalc/}}\end{footnote}

\section{Preliminaries on equivariant K theory}\label{sec:preliminaries} In this section we recall some basic facts about the equivariant 
K-theory of a variety with a group action. For an introduction to equivariant K theory, and more details, see \cite{chriss2009representation}. 

Let $X$ be a smooth projective variety with an action of a linear algebraic group $G$. The equivariant K theory ring $\K_G(X)$ is the Grothendieck ring generated by symbols $[E]$, where $E \to X$ is an $G$-equivariant vector bundle, modulo the relations $[E]=[E_1]+[E_2]$ for any short exact sequence $0 \to E_1 \to E \to E_2 \to 0$ of equivariant vector bundles. The additive ring structure is given by direct sum, and the multiplication is given by tensor products of vector bundles. 
Since $X$ is smooth, any $G$-linearized coherent sheaf has a finite resolution by (equivariant) vector bundles, and the ring $\K_G(X)$ coincides with the Grothendieck group of $G$-linearized coherent sheaves on $X$. In particular, any 
$G$-linearized coherent sheaf $\mathcal{F}$ on $X$ determines a class $[\mathcal{F}] \in \K_G(X)$. 
An important special case is if $\Omega \subset X$ is a $G$-stable subscheme; then its structure sheaf 
determines a class $[\cO_\Omega] \in \K_G(X)$.

The ring $\K_G(X)$ is an algebra over $\K_G(pt) = \Rep(G)$, the representation ring of 
$G$. If $G=T$ is a complex torus, then this is the Laurent polynomial ring 
$\K_T(pt) = \Z[T_1^{\pm 1}, \ldots, T_n^{\pm 1}]$ where 
$T_i:=e^{t_i}$ are characters corresponding to a basis of the 
Lie algebra of $T$. 

The (Hirzebruch) $\lambda_y$ class is defined by 
\[ \lambda_y(E) := 1 + y [E] + y^2 [\wedge^2 E] + \ldots + y^e  [\wedge^e E] \in \K_G(X)[y] \/. \] 
This class was introduced by Hirzebruch \cite{hirzebruch:topological} in order to help with the 
formalism of the Grothendieck-Riemann-Roch theorem. It may be thought as the 
K theoretic analogue of the (cohomological) Chern polynomial 
\[ c_y(E)= 1+ c_1(E) y + \ldots + c_e(E) y^e\] of the bundle $E$. 
The $\lambda_y$ class is multiplicative with respect to short exact sequences, i.e. if 
\[ \begin{tikzcd} 0 \arrow[r] & E_1 \arrow[r] & E_2 \arrow[r] & E_3 \arrow[r] & 0 \end{tikzcd} \] 
is such a sequence of vector bundles then 
\[ \lambda_y(E_2) = \lambda_y(E_1) \cdot \lambda_y(E_3) \/; \] 
cf.~\cite{hirzebruch:topological}. 
A particular case of this construction is when $V$ is a (complex) vector space with an action 
of a complex torus $T$, 
and with weight decomposition 
$V = \oplus_i V_{\mu_i}$, where each $\mu_i$ is a weight in the dual of the Lie algebra of $T$. 
The {\em character} of $V$ is the element $\ch_T(V):=\sum_i \dim V_{\mu_i} e^{\mu_i}$, regarded  in $\K_T(pt)$.
The  $\lambda_y$ class of $V$ is the element 
$\lambda_y(V) =\sum_{i \ge 0} y^i ch(\wedge^i V) \in \K_T(pt)[y]$.
From the multiplicativity property of the $\lambda_y$ class it follows that 
\[\lambda_y(V) = \prod_i (1+y e^{\mu_i})^{\dim V_{\mu_i}} \/;\] 
see \cite{hirzebruch:topological}.

Since $X$ is proper, the push-forward to a point equals the Euler 
characteristic, or, equivalently, the virtual representation, 
\[\chi(X, \mathcal{F})= \int_X [\mathcal{F}] := \sum_i (-1)^i \ch_T H^i(X, \mathcal{F}) \/. \] 
For $E,F$ equivariant vector bundles, this gives a pairing 
\[ \langle - , - \rangle :\K_G(X) \otimes \K_G(X) \to \K_G(pt); \quad
\langle [E], [F] \rangle := \int_X E \otimes F = \chi(X, E \otimes F) \/. \]

\section{Equivariant K theory of flag manifolds} 
In this paper we will utilize the torus equivariant K-theory of the partial flag 
manifolds $\Fl(i_1, i_2, \ldots, i_k;n)$, which parametrize partial flags 
$F_{i_1} \subset \ldots \subset F_{i_p} \subset \C^n$, where $\dim F_i = i$, to study
the equivariant quantum K theory ring of Grassmann manifolds $\Gr(k;n)$.
The goal of this section is to review some of the basic features of the equivariant
K rings of the partial flag manifolds. Of special importance is the calculations of K theoretic 
push forwards of Schur bundles, which will play a key role later in the paper. 
 
\subsection{Basic definitions} The flag manifold $\Fl(i_1, i_2, \ldots, i_k;n)$ is an algebraic variety
homogeneous under the action of $\mathrm{GL}_n:=\mathrm{GL}_n(\C)$. 
Let $T \subset \mathrm{GL}_n$ be the subgroup
of diagonal matrices acting coordinate-wise on $\C^n$. Denote by 
$T_i \in \K_T(pt)$ the weights of this action.

Set $W$ to be the symmetric group in $n$ letters, and let $W_{i_1, \ldots, i_p} \le W$ be 
the subgroup generated by simple reflections $s_i = (i,i+1)$ where $i \notin \{ i_1, \ldots, i_p\}$. 
Denote by $\ell:W \to \mathbb{N}$ the length function, and by  
$W^{i_1, \ldots, i_k}$ the set of minimal length representatives of $W/W_{i_1, \ldots, i_p}$. This consists
of permutations $w \in W$ which have descents at positions $i_1, \ldots, i_p$, i.e., $w(i_k+1)< \ldots <w(i_{k+1})$, for $k=0, \ldots, p$, with the convention that $i_0=1, i_{p+1} =n$. 

The torus fixed points $e_w \in \Fl(i_1, i_2, \ldots, i_k;n)$ are indexed by the permutations 
$w \in W^{i_1, \ldots, i_p}$. For each such point, the {\em Schubert cell} 
$X_{w}^\circ \subset \Fl(i_1, i_2, \ldots, i_k;n)$ is the orbit $B^-. e_w$ of the Borel subgroup of lower triangular matrices, and the {\em Schubert variety} is the closure $X_w = \overline{X^{w,\circ}}$. Our convention ensures that $X_w$ has codimension $\ell(w)$. The Schubert varieties $X^w$ determine classes $\mathcal{O}_w:= [\mathcal{O}_{X_w}]$ 
in $\K_T(X)$. Similarly, the $S$-fixed points give classes denoted by 
$\iota_w:= [\mathcal{O}_{e_w}]$. The equivariant K-theory $\K_T(\Fl(i_1, \ldots, i_p;n))$ is a free module over 
$\K_T(pt)$ with a basis given by Schubert classes $\{\mathcal{O}_w \}_{w \in W^{i_1, \ldots, i_p}}$. 
For the Grassmannian $\Gr(k;n)$, the Schubert varieties are indexed by partitions 
$\lambda= (\lambda_1, \ldots, \lambda_k)$ included in 
the $k \times (n-k) $ rectangle, i.e., $n-k \ge \lambda_1 \ge \ldots \ge \lambda_k \ge 0$. For Grassmanians,
we denote the Schubert variety by $X_\lambda$; this has codimension $|\lambda| = \lambda_1 + \ldots + \lambda_k$.
The corresponding class in $\K_T(\Gr(k;n))$ is denoted $\cO_\lambda:= [\cO_{X_\lambda}]$. 

Let ${\bf i} = (1 < i_1<  \ldots <i_p < n)$ and ${\bf j}$ obtained from 
${\bf{i}}$ by removing some of the indices $i_s$. Denote by $\Fl({\bf i})$ and $\Fl({\bf j})$ the 
respective flag manifolds and by $\pi_{{\bf i},{\bf j}}: \Fl({\bf i}) \to \Fl({\bf j})$ the natural projection. 
We will need the following well known fact:
\begin{lemma}\label{lemma:proj} Let $w \in W^{{\bf i}}$. Then for any Schubert varieties
$X_w \in \Fl({\bf i})$ and $X_v \in \Fl({\bf j})$
\[ (\pi_{{\bf i},{\bf j}})_*\mathcal{O}_w = \mathcal{O}_{w'} \in K_T(\Fl({\bf j})) \/, \textrm{ and } 
\pi_{{\bf i},{\bf j}}^*\mathcal{O}_v = \mathcal{O}_{v} \in K_T(\Fl({\bf i})) \]
where $w' \in W^{{\bf j}}$ is the minimal length representative for $w \in W/W_{{\bf j}}$.
\end{lemma}
\begin{proof} The first equality follows from \cite[Thm. 3.3.4(a)]{brion.kumar:frobenius} and the second 
because $\pi_{{\bf i},{\bf j}}$ is a flat morphism.\end{proof}

\subsection{Push-forward formulae of Schur bundles} Next we recall some results about cohomology of 
Schur bundles on Grassmannians. Our main reference is Kapranov's paper \cite{kapranov:Gr}. A reference
for basic definitions of Schur bundles is Weyman's book \cite{weyman}.

Recall that if $X$ is a $T$-variety, $\pi:E \to X$ is any $T$-equivariant vector bundle of rank $e$, 
and $\lambda=(\lambda_1, \ldots, \lambda_k)$ 
is a partition with at most $e$ parts, the {\em Schur bundle} $\mathfrak{S}_\lambda(E)$ is a 
$T$-equivariant vector bundle over $X$. It has the property
that if $x \in X$ is a $T$-fixed point, the fibre $(\fS_\lambda(E))_x$ 
is the $T$-module with character the Schur function $s_\lambda$. For example,
if $\lambda = (1^k)$, then $\fS_{(1^k)}(E) = \wedge^k E$, and if $\lambda = (k)$ then
$\fS_{(k)}(E)=\mathrm{Sym}^k(E)$.

In this paper $X=\Gr(k;\C^n)$ with the $T$-action restricted from
$\mathrm{GL}_n(\C)$. 
To emphasize the $T = (\C^*)^n$-module structure on
$\C^n$, we will occasionally denote by $V:=\C^n$ and
by $\Gr(k;V)$ the corresponding Grassmannian.
Further, $V$ will also be identified with the trivial, 
but not equivariantly trivial, vector bundle $\Gr(k;V) \times V$.
The following was proved in \cite{kapranov:Gr}.
\begin{prop}[Kapranov]\label{lemma:BWB} Consider the 
Grassmannian $\Gr(k;V)$ with the tautological sequence 
$0 \to \cS \to V \to \cQ \to 0$. 
For any nonempty partition {$\lambda= (\lambda_1\geq \lambda_2 \geq \ldots \geq \lambda_k \geq 0)$ such that 
$\lambda_1 \le n-k$}, 
there are the following isomorphisms of $T$-modules:

(a) For all $i \ge 0$, $H^i(\Gr(k;V), \fS_\lambda(\cS)) = 0$.

(b) \[ H^i(\Gr(k;V), \fS_\lambda(\mathcal{S}^*)) = \begin{cases} \fS_\lambda(V^*) & i=0 \\ 0 & i>0 \/. \end{cases}\]

(c) \[ H^i(\Gr(k;V), \fS_\lambda(\cQ)) = \begin{cases} \fS_\lambda(V) & i=0 \\ 0 & i>0 \/. \end{cases}\]

\end{prop}
\begin{proof} Parts (a) and (b) were proved in \cite[Prop. 2.2]{kapranov:Gr}, as a consequence of
the Borel-Weil-Bott theorem on the complete flag manifold.
For part (c), observe that there is a $T$-equivariant isomorphism 
$\Gr(k;V) \simeq \Gr(\dim V - k; V^*)$ 
under which the $T$-equivariant bundle $\cQ$ is sent to $\cS^*$. Then part (c) follows from (b).
\end{proof} 

We also need a relative version of \Cref{lemma:BWB}. Consider an $T$-variety $X$ equipped with 
an $T$-equivariant
vector bundle $\mathcal{V}$ of rank $n$. Denote by 
$\pi: \mathbb{G}(k,\mathcal{V}) \to X$ the Grassmann bundle over $X$. It is equipped with a 
tautological sequence
$0 \to \underline{\cS} \to \pi^* \mathcal{V} \to \underline{\cQ} \to 0$ over $\mathbb{G}(k,\mathcal{V})$. The following corollary follows from \Cref{lemma:BWB},
using that $\pi$ is a $T$-equivariant locally trivial fibration:

\begin{cor}\label{cor:relativeBWB}
For any nonempty partition {$\lambda= (\lambda_1\geq \lambda_2 \geq \ldots \geq \lambda_k \geq 0)$ such that 
$\lambda_1 \le n-k$}, there are the following isomorphisms of $T$-modules:

(a) For all $i \ge 0$, the higher direct images, $R^i \pi_* \fS_\lambda(\underline{\cS}) = 0$. 

(b) \[ R^i \pi_*\fS_\lambda(\underline{\cS}^*)) = \begin{cases} \fS_\lambda(\mathcal{V}^*) & i=0 \\ 0 & i>0 \/. \end{cases}\]

(c) \[ R^i \pi_*\fS_\lambda(\underline{\cQ}) = \begin{cases} \fS_\lambda(\mathcal{V}) & i=0 \\ 0 & i>0 \/. \end{cases}\]
\end{cor}

\section{Quantum K theory and `quantum=classical'}
\subsection{Definitions and notation} The quantum K ring was defined by Givental and Lee \cite{givental:onwdvv, lee:QK}. 
We recall 
the definition below, following \cite{givental:onwdvv}.~The {\em quantum K metric} 
is a deformation of the usual K-theory pairing; we recall the definition of this metric for $X= \Gr(k;n)$. 
Fix any basis $(a_\lambda)$ of $\K_T(\Gr(k;n))$ over $\K_T(pt)$, 
where $\lambda$ varies over partitions in the $k \times (n-k)$ rectangle.
The {\em (small) quantum K metric} is defined by 
\begin{equation}\label{E:qKmetric} (a_\lambda , a_\mu)_{sm} 
= \sum_{d \ge 0} q^d \langle a_\lambda, a_\mu \rangle_d \quad \in \K_T(pt)[[q]]\/,\end{equation} 
extended by $K_T(pt)[[q]]$-linearity. The elements 
$\langle a_\lambda, a_\mu \rangle_d \in \K_T(pt)$ denote
$2$-point (genus $0$, equivariant) K-theoretic Gromov-Witten (KGW) invariants if $d>0$; if $d=0$ then 
$\langle a_\lambda, a_\mu \rangle_0= \langle a_\lambda, a_\mu\rangle$ are given 
by the ordinary K-theory pairing. The ($n$-point, genus $0$) KGW invariants 
$\langle a_1, \ldots, a_n \rangle_d \in \K_T(pt)$
are defined by pulling
back via the evaluation maps, then integrating, over the Kontsevich moduli space of stable maps 
$\overline{\mathcal{M}}_{0,n}(\Gr(k;n),d)$. Instead of recalling the precise definition of the KGW invariants, 
in \Cref{thm:q=cl} below
we give a `quantum=classical' statement calculating $2$ and $3$-point KGW invariants of Grassmannians.  
Explicit combinatorial formulae for the $2$-point KGW invariants for any homogeneous space 
may be found in \cite{buch.m:nbhds,BCLM:euler}.

\begin{example} Consider the projective plane $\bP^2$, and for simplicity work non-equivariantly. 
Consider the Schubert basis $\cO_0= [\cO_{\bP^2}], \cO_1= [\cO_{line}], \cO_2=[\cO_{pt}]$. 
The classical K-theory pairing gives the matrix 
\[ (\langle \cO_i, \cO_j \rangle) = \begin{pmatrix} 1 & 1 & 1 \\ 1 & 1 & 0 \\ 1 & 0 & 0 \end{pmatrix} \] 
For any $i,j \ge 0$ and $d >0$, $\langle \cO_i, \cO_j \rangle_d =1$. Then the quantum K metric is given by
\[ (\cO_i, \cO_j )_{sm} = \langle \cO_i, \cO_j \rangle + q+q^2 + ... = \langle \cO_i, \cO_j \rangle + \frac{q}{1-q} \/.\] 
\end{example} 

\begin{thm}[Givental \cite{givental:onwdvv}] Consider the $\K_T[pt][[q]]$-module 
\[ \QK_T(\Gr(k;n)):= \K_T(\Gr(k;n)) \otimes \K_T(pt)[[q]] \/. \]
Define the (small) equivariant quantum K product $\star$ on $\QK_T(\Gr(k;n))$ by the equality
\begin{equation}\label{E:QKdefn}(a \star b, c)_{sm} = \sum_{d \ge 0} q^d \langle a,b,c \rangle_d \/,\end{equation}
for any $a, b, c \in \K_T(\Gr(k;n))$. Then $(\QK_T(\Gr(k;n)), +,\star)$ is a commutative, associative $\K_T(pt)[[q]]$-algebra. Furthermore, the small quantum K metric
gives it a structure of 
a Frobenius algebra, i.e. $(a \star b, c)_{sm} = (a,b \star c)_{sm} $.
\end{thm}
\begin{remark} It was proved in \cite{BCMP:qkfin} that the submodule 
$\K_T(\Gr(k;n)) \otimes \K_T(pt)[q]$ is stable under the QK product $\star$. This means 
that the product of two Schubert classes in $\QK_T(\Gr(k;n))$ 
has structure constants which are polynomials in 
$q$. Similar statements hold for any flag manifold \cite{ACTI:finiteness,kato:loop}. 
However, working over the ring of formal power series in $q$ {is needed for 
proving module finite generation of the claimed presentations. It is also}
natural if one studies quantizations of dual bundles,
such as those from \Cref{thm:qdual} below.
\end{remark}

We record the following immediate consequence.
\begin{cor}\label{lemma:QKeq} Let $a,b,c \in \K_T(\Gr(k;n))[[q]]$. Then $a \star b = c$ if and only if for any 
$\kappa \in \K_T(\Gr(k;n))$, 
\[ \sum_{d \ge 0} q^d \langle a,b,\kappa \rangle_d = \sum_{d \ge 0} q^d \langle c , \kappa \rangle_d \/. \]
\end{cor}

\subsection{Quantum=classical} The goal of this section is to recall the `quantum = classical' statement,
which relates the ($3$-point, genus $0$) equivariant KGW invariants on Grassmannians to a `classical' calculation in the equivariant K-theory of a two-step flag manifold. This statement was proved in \cite{buch.m:qk}, and it generalized 
results of Buch, Kresch and Tamvakis \cite{buch.kresch.ea:gromov-witten} from quantum cohomology. The proofs
rely on the `kernel-span' technique introduced by Buch \cite{buch:qcohgr}. A Lie-theoretic approach, 
for large degrees $d$, and for the larger family of cominuscule Grassmannians, was obtained in 
\cite{chaput.perrin:rationality}; see also \cite{bcmp:projgw} for an alternative way to `quantum=classical' utilizing projected 
Richardson varieties. 

We recall next the `quantum = classical' result, proved in \cite{buch.m:qk}, and 
which will be used later in this paper. To start, form the 
following incidence diagram:
\begin{equation}\label{E:qcl-all}\begin{CD}  Z_d:=\mathrm{Fl}(k-d, k , k+d; n) @>p_1' >> \mathrm{Fl}(k-d,k;n) 
@>{\hspace{0.5cm p}\hspace{0.5cm}}>>X:=\Gr(k;n) \\ @V{p_2}VV @V{q}VV @. \\  
Y_d:=\mathrm{Fl}(k-d,k+d;n) @>pr>>\mathrm{Gr}(k-d;n)@.\end{CD} \end{equation}
Here all maps are the natural projections. Denote by $p_1: \Fl(k-d,k,k+d;n) \to \Gr(k;n)$
the composition $p_1:=p \circ p_1'$. If $d\ge k$ then we set $Y_d:= \Fl(k+d;n)$ and
if $k+d \ge n$ then we set $Y_d:= \Gr(k-d;n)$. In particular, if $d \ge \min \{ k, n-k \}$, then
$Y_d$ is a single point.

\begin{thm}\label{thm:q=cl} Let $a,b,c \in K_T(\Gr(k;n))$ and $d \ge 0$ a degree.

(a) The following equality holds in $K_T(pt)$:
\[ \langle a, b, c \rangle_d = \int_{Y_d} (p_2)_*(p_1^*(a))\cdot (p_2)_*(p_1^*(b))\cdot (p_2)_*(p_1^*(c)) \/. \]

(b) Assume that $(p_2)_*(p_1^*(a)) = pr^*(a')$ for some $a' \in K_T(\Gr(k-d;n)$. Then 
\[ \langle a, b, c \rangle_d = \int_{\Gr(k-d;n)} a' \cdot (q)_*(p^*(b))\cdot (q)_*(p^*(c)) \/. \]
\end{thm}
Observe that part (b) follows from (a) and the fact that the left diagram is a fibre square; for details
see \cite{buch.m:qk}. We will often use the tautological sequence on $\Gr(k;n)$: 
\[ 0 \to \cS \to  \mathbb{C}^n \to  \cQ:=\mathbb{C}^n/\cS \to 0 \/. \] 
To lighten notation, we will denote by the same letters the bundles on various flag manifolds
from the diagram \eqref{E:qcl-all}, but we will indicate in the subscript the rank of the bundle in question.
To illustrate in a situation used below, observe that 
$\mathrm{Fl}(k-d,k;n)$ equals to the Grassmann bundle 
$\mathbb{G}(d, \mathbb{C}^n/\cS_{k-d}) \to \Gr(k-d;n)$, 
with tautological sequence 
\[ 0 \to \cS_k/\cS_{k-d} \to  \mathbb{C}^n/\cS_{k-d} \to  \mathbb{C}^n/\cS_{k} \simeq p^*(\mathbb{C}^n/\cS_{k}) \to 0 \/; \]
all these are bundles on $\mathrm{Fl}(k-d,k;n)$, with $\cS_k$ is pulled back from $\Gr(k;n)$, and 
$\cS_{k-d}$ is pulled back from $\Gr(k-d;n)$. 

\section{Cohomological calculations from `quantum = classical' diagrams} In this section 
we calculate some push-forwards of tautological bundles, which will be utilized in the proof of our main result.
For a vector bundle $E$ of rank $e$, we denote by $\lambda(E)_{\ge s}$, 
respectively $\lambda(E)_{\le s}$ and so on, the truncations 
$y^s [\wedge^s E] + \ldots + y^e [\wedge^e E]$, respectively 
$1 + y [E] + \ldots + [\wedge^{s} E]$ etc. 
We use the notation from the diagram \eqref{E:qcl-all}.
\begin{prop}\label{prop:push-pull} Let $d \ge 1$. Then the following hold:

(a) For any $j \ge 0$, $(p_2)_*p_1^* (\wedge^j \cS) = \wedge^j  \cS_{k-d}$ in $\K_T(Y_d)$. In particular, 
\begin{equation}\label{E:pfS} (p_2)_*(p_1^* \lambda_y(\cS)) = \lambda_y(\cS_{k-d}) \/. \end{equation}

(b) The following holds in $\K_T(Y_d)$:
\begin{equation}\label{E:pfQ} (p_2)_*p_1^* (\lambda_y (\cQ)) = \lambda_y(\C^n/\cS_{k+d}) \cdot \lambda_y(\cS_{k+d}/\cS_{k-d})_{\le d} \/. \end{equation} 

(c) The following holds in $\K_T(Y_d)$:
\begin{equation}\label{E:pfSvee} (p_2)_*p_1^* (\lambda_y (\cS^*)) = \lambda_y(\cS_{k-d}^*) \cdot 
\lambda_y((\cS_{k+d}/\cS_{k-d})^*)_{\le d} \/. \end{equation}
\end{prop}
\begin{proof} Consider the exact sequence on $\Fl(k-d,k,k+d)$ given by \[ 0 \to \cS_{k-d} \to S_k = p_1^*\cS \to \cS_k/\cS_{k-d} \to 0 \/. \]
It follows that $\lambda_y(\cS_k) = \lambda_y(\cS_{k-d}) \cdot \lambda_y(\cS_k/\cS_{k-d})$. By the projection formula, 
\begin{equation}\label{E:p2lyS} (p_2)_*(\lambda_y(\cS_k)) = \lambda_y(\cS_{k-d}) \cdot (p_2)_*(\lambda_y(\cS_k/\cS_{k-d})) \/. \end{equation}
Observe now that $p_2: \Fl(k-d,k,k+d;n) \to \Fl(k-d,k+d;n)$ may be identified with the Grassmann bundle
$\mathbb{G}(d,\cS_{k+d}/\cS_{k-d})$, with tautological subbundle $\cS_k/\cS_{k-d}$. Then part (a) follows from 
the projection formula and \Cref{cor:relativeBWB}, which shows that
\[ (p_2)_*(\lambda_y(\cS_k/\cS_{k-d})) = [\cO_{\Fl(k-d,k+d;n)}] \/.\] 


{For (b)}, consider the short exact sequence on $\Fl(k-d,k,k+d;n)$, 
\[ 0 \to \cS_{k+d}/\cS_k \to \C^n/\cS_k = p_1^*\cQ \to \C^n/\cS_{k+d} \to 0 \/. \] 
Then $\lambda_y(p_1^* \cQ) = \lambda_y(\cS_{k+d}/\cS_k) \cdot (p_2)^*\lambda_y(\C^n/\cS_{k+d})$, {and 
by the projection formula 
\[ (p_2)_*(\lambda_y(p_1^* \cQ)) = (p_2)_*(\lambda_y(\cS_{k+d}/\cS_k)) \cdot \lambda_y(\C^n/\cS_{k+d}) \/. \]
}The claim follows again by \Cref{cor:relativeBWB}, {using that 
$\cS_{k+d}/\cS_k$ is the tautological quotient bundle of the Grassmann bundle 
$\mathbb{G}(d,\cS_{k+d}/\cS_{k-d})$, thus}
\[ \pi_*\lambda_y(\cS_{k+d}/\cS_k) = 1+ y[\cS_{k+d}/\cS_{k-d}] + 
\ldots + y^d \wedge^d [\cS_{k+d}/\cS_{k-d}] \/. \]  

Part (c) follows similarly to (b), utilizing the exact sequence 
$0 \to (\cS_k/\cS_{k-d})^* \to \cS_k^* = p_1^*(\cS^*) \to \cS_{k-d}^* \to 0$
and \Cref{cor:relativeBWB}.\end{proof}
We need the analogues of some of the previous equations in the equivariant K theory of $\Gr(k-d;n)$.
\begin{cor}\label{cor:pfinGr} (a) The following equations hold in $\K_T(\Gr(k-d;n))$:

(a) $q_*(p^* \lambda_y(\cS)) = \lambda_y(\cS_{k-d})$.

(b) $q_*p^* (\lambda_y (\cS^*)) = \lambda_y(\cS_{k-d}^*) \cdot \lambda_y((\C^n/\cS_{k-d})^*)_{\le d}$.

(c) $q_*p^* (\lambda_y (\cQ)) = \lambda_y(\C^n/\cS_{k-d})_{\le n-k}$.
\end{cor}
\begin{proof} Parts (a) and (b) follow similarly to \Cref{prop:push-pull} (a). For (a), we utilize the short exact sequence
$0 \to \cS_{k-d} \to S_k = p_1^*\cS \to \cS_k/\cS_{k-d} \to 0$ on $\Fl(k-d, k;n)$, and that the projection $q:\Fl(k-d,k;n) \to \Gr(k-d;n)$ is identified to the Grassmann bundle $\mathbb{G}(d;\C^n/\cS_{k-d}) \to \Gr(k-d;n)$ with tautological subbundle $\cS_k/\cS_{k-d}$. For (b) one takes the dual of this sequence. Then by \Cref{cor:relativeBWB}
\[ q_*(\lambda_y(\cS_{k}/\cS_{k-d})) = [\cO_{\Gr(k-d;n)}] \textrm{ and } 
q_*(\lambda_y((\cS_{k}/\cS_{k-d})^*))= \lambda_y((\C^n/\cS_{k-d})^*)_{\le d} \/.\] 
Then (a) and (b) follow from this and the projection formula.

For part (c), observe that $\Fl(k-d,k;n) \to \Gr(k-d;n)$ is the Grassmann bundle $\mathbb{G}(d,\C^n/\cS_{k-d})$
with tautological quotient bundle $\C^n/\cS_k = p^*\cQ$. Then the claim follows again by \Cref{cor:relativeBWB}.
\end{proof}
Observe that $pr^*(\lambda_y(\cS_{k-d})) = \lambda_y(\cS_{k-d})$ in $\K_T(Y_d)$. Utilizing this, and combining \Cref{prop:push-pull}, 
\Cref{cor:pfinGr} and part (b) from \Cref{thm:q=cl} we obtain the following useful corollary:

\begin{cor}\label{cor:KGWS} Fix arbitrary $b,c \in \K_T(\Gr(k;n))$ and any degree $d \ge 0$. Then
the equivariant KGW
invariant $\langle \lambda_y(\cS),b,c \rangle_d$ satisfies:
\[ \langle \lambda_y(\cS),b,c \rangle_d = \int_{\Gr(k-d;n)} \lambda_y(\cS_{k-d}) \cdot q_*p^*(b) \cdot q_*p^*(c) \/. \]
In particular, the $2$-point KGW invariant $\langle b,c \rangle_d$ satisfies:
\[ \langle b,c \rangle_d = \int_{\Gr(k-d;n)} q_*p^*(b)\cdot q_*p^*(c) \/. \]
\end{cor}

\section{Quantum duals}\label{sec:qduals} The main goal of this section is to prove the next identity.
\begin{thm}\label{thm:qdual}The following holds in $\QK_T(\Gr(k;n))$:
\[ (\lambda_y(\cS)-1) \star \det (\cQ) = (1-q) ((\lambda_y(\cS) -1) \cdot \det (\cQ)) \/. \]
Equivalently, for any $i>0$, 
\[ \wedge^i (\cS) \star \det (\cQ) = (1-q) \wedge^{k-i}(\cS^*) \cdot \det(\C^n) \/. \]
\end{thm}
An equivalent formulation of this identity is in terms of (equivariant) KGW invariants: for any 
degree $d \ge 0$ and any $a \in \K_T(\Gr(k;n))$, 
\begin{equation}\label{E:qdualGW} \langle  \lambda_y(\cS) -1, \det (\cQ) , a \rangle_d = \langle (\lambda_y(\cS) -1) \cdot \det (\cQ), a \rangle_d - 
\langle (\lambda_y(\cS)-1) \cdot \det (\cQ), a \rangle_{d-1} \/. 
\end{equation}
By \Cref{cor:KGWS} and \Cref{cor:pfinGr}, the left hand side may be calculated as:
\begin{equation}\label{E:LHSqdual} 
 \langle  \lambda_y(\cS)-1, \det (\cQ) , a \rangle_d = \int_{\Gr(k-d;n)} (\lambda_y(\cS_{k-d})-1) \cdot \wedge^{n-k} (\C^n/\cS_{k-d}) \cdot q_* p^*(a) \/. \end{equation}
The next lemma calculates the push-forwards relevant for the right hand side.
\begin{lemma}\label{lemma:pfGr} The following equality holds in $\K_T(\Gr(k-d;n))$: 
\[ y^{n-k} q_*(p^*(\lambda_y(\mathcal{S}_k) \cdot \det (\cQ) )) = 
 \lambda_y(\mathcal{S}_{k-d}) \cdot \lambda_y(\C^n/\cS_{k-d})_{\ge n-k}
 \/.\]
\end{lemma}
\begin{proof} To start, observe that for any (equivariant) vector bundle $E$,
\[ \wedge^i E \otimes \det(E^*) = \wedge^{rk(E)-i} E^* \quad \textrm{ and } \quad 
y^{rk(E)} \lambda_{y^{-1}}(E^*) = \lambda_y(E) \cdot \det(E^*) \/. 
\]
Utilizing this and that $\det (\C^n/\cS_k) = \det \cS_k^* \otimes \det \C^n$ , we obtain that 
\[ 
\begin{split} 
& y^{n-k} q_*(p^*(\lambda_y(\mathcal{S}_k) \cdot \det (\C^n/\mathcal{S}_k))) \\
& = y^{n-k} q_*(p^*(\lambda_y(\mathcal{S}_k) \cdot \det \mathcal{S}_k^* \cdot \det \C^n))\\
& = y^{n-k} y^{k} q_*(\lambda_{y^{-1}}(\mathcal{S}_k^*) \cdot \det \C^n)) \\
& = y^n q_*(\lambda_{y^{-1}}(\mathcal{S}_{k-d}^*) \cdot \lambda_{y^{-1}}(\mathcal{S}_{k}/\cS_{k-d})^*) \cdot \det \C^n \\
& = y^n \lambda_{y^{-1}}(\cS_{k-d}^*) \cdot q_*(\lambda_{y^{-1}}(\mathcal{S}_{k}/\cS_{k-d})^*) \cdot \det \C^n\/.
\end{split} 
\]
By \Cref{cor:relativeBWB}, $q_*(\lambda_y((\cS_{k}/\cS_{k-d})^*))= \lambda_y((\C^n/\cS_{k-d})^*)_{\le d}$, therefore
\[  \begin{split} y^d & q_*(\lambda_{y^{-1}}(\mathcal{S}_{k}/\cS_{k-d})^*) 
 = y^d (1+ y^{-1} (\C^n/\cS_{k-d})^* + \ldots + y^{-d} \wedge^d (\C^n/\cS_{k-d})^*) \\
& = \det (\C^n/\cS_{k-d})^* \cdot ( y^d \wedge^{n-k+d} (\C^n/\cS_{k-d})  \\ & \quad + y^{d-1} \wedge^{n-k+d-1}( \C^n/\cS_{k-d}) + 
\ldots + \wedge^{n-k} (\C^n/\cS_{k-d}) ) \\
& = y^{-(n-k)} \det (\C^n/\cS_{k-d})^* \cdot \lambda_y( \C^n/\cS_{k-d} )_{\ge n-k} \/. \end{split} \]
Combining the last two equalities above we obtain:
\[ \begin{split} 
& y^{n-k} q_*(p^*(\lambda_y(\mathcal{S}_k) \cdot \det (\C^n/\mathcal{S}_k)) \\
& = y^n \lambda_{y^{-1}}(\cS_{k-d}^*) \cdot q_*(\lambda_{y^{-1}}(\mathcal{S}_{k}/\cS_{k-d})^*) \cdot \det \C^n
\\
& = y^{n-d} \lambda_{y^{-1}}(\cS_{k-d}^*) \cdot y^{-(n-k)} \det (\C^n/\cS_{k-d})^* \cdot \lambda_y( \C^n/\cS_{k-d} )_{\ge n-k} 
\cdot \det \C^n\\ &
= y^{k-d} \lambda_{y^{-1}}(\cS_{k-d}^*) \cdot \det (\cS_{k-d}) \cdot \lambda_y( \C^n/\cS_{k-d} )_{\ge n-k} 
\\
& =\lambda_y(\cS_{k-d}) \cdot \lambda_y( \C^n/\cS_{k-d} )_{\ge n-k} \/. 
\end{split}
\]
The last expression is the one from the claim.\end{proof}
Consider the following projection maps:
\begin{equation}\label{E:fibrediag}
\begin{tikzcd} & \mathrm{Fl}(k-d,k-d+1,k;n) \ar[r,"p^{k-d}"] \ar[ld,"q^{k-d+1}"] \ar[dr,swap,"q^{k-d}"] & \Gr(k;n) \\
\mathrm{Gr}(k-d+1;n)& & \mathrm{Gr}(k-d;n)
\end{tikzcd}
\end{equation}
We are now ready to prove \Cref{thm:qdual}.
\begin{proof}[Proof of \Cref{thm:qdual}] A standard diagram chase and the equalities from \Cref{E:LHSqdual} and \Cref{lemma:pfGr} imply that in order to prove
\Cref{E:qdualGW} it suffices to show that in $\K_T(\Fl(k-d,k-d+1,k;n)$,
\[ \begin{split} 
y^{n-k} & p^{k-d}_*((\lambda_y(\cS_{k-d})-1) \cdot \wedge^{n-k}(\C^n/\cS_{k-d})) \\
 & 
 = p^{k-d}_*\Bigl(\lambda_y(\cS_{k-d}) \cdot \lambda_y(\C^n/\cS_{k-d})_{\ge n-k} 
 - \lambda_y(\cS_{k-d+1}) \cdot \lambda_y(\C^n/\cS_{k-d+1})_{\ge n-k}\Bigr) \\
 & 
 \quad - y^{n-k}\Bigl(\wedge^{n-k}(\C^n/\cS_{k-d}) - \wedge^{n-k} (\C^n/\cS_{k-d+1})\Bigr) \/.\end{split} \]
After expanding and canceling the like terms, this amounts to showing that
\begin{equation}\label{E:qdualfinal} 
\begin{split} \quad \quad p^{k-d}_*(\lambda_y(\cS_{k-d}) \cdot \lambda_y(\C^n/\cS_{k-d})_{> n-k}) & = p^{k-d}_*(\lambda_y(\cS_{k-d+1})\cdot \lambda_y(\C^n/\cS_{k-d+1})_{\ge n-k}) \\ & - y^{n-k}p^{k-d}_*(\wedge^{n-k} (\C^n/\cS_{k-d+1}) )\/. 
\end{split} \end{equation}
Since this identity will be used again in the proof of relations in the quantum K ring, we will
prove it in the lemma below, thus finishing the proof of \Cref{thm:qdual}.\end{proof}  
\begin{lemma}\label{lemma:maineq} \Cref{E:qdualfinal} holds.\end{lemma}
\begin{proof} Consider the projection $p': \Fl(k-d,k-d+1,k;n) \to \Fl(k-d+1,k;n)$; this is the 
Grassmann bundle $\mathbb{G}(k-d;\cS_{k-d+1})$ over $\Fl(k-d+1,k;n)$ with tautological subbundle 
$\cS_k$ and quotient bundle $\mathcal{L} := \cS_{k-d+1}/\cS_{k-d}$. In order to show the equality \ref{E:qdualfinal}, it suffices to replace
$p^{k-d}$ by $p'$, i.e., 
\begin{equation}\label{E:qdualfinal2} \begin{split} p'_*(\lambda_y(\cS_{k-d}) \cdot \lambda_y(\C^n/\cS_{k-d})_{> n-k}) &= p'_*(\lambda_y(\cS_{k-d+1}) \cdot \lambda_y(\C^n/\cS_{k-d+1})_{\ge n-k}) \\ & - y^{n-k}p'_*(\wedge^{n-k} (\C^n/\cS_{k-d+1}) )\/. \end{split}\/. \end{equation}
To calculate the left hand side, consider the short exact sequences 
\[ 0 \to \mathcal{S}_{k-d} \to \mathcal{S}_{k-d+1} \to \mathcal{L} \to 0 \/; \quad  
0 \to \mathcal{L} \to \C^{n}/\mathcal{S}_{k-d} \to \C^{n}/\mathcal{S}_{k-d+1} \to 0 \/. \] From this it follows that 
$\lambda_y(\C^{n}/\mathcal{S}_{k-d}) = (1+y \mathcal{L}) \lambda_y(\C^{n}/\mathcal{S}_{k-d+1})$, thus
\begin{equation*}\label{E:100} \lambda_y(\C^{n}/\mathcal{S}_{k-d})_{>n-k} =  \lambda_y(\C^{n}/\mathcal{S}_{k-d+1})_{>n-k} + y 
\mathcal{L} \cdot  \lambda_y(\C^{n}/\mathcal{S}_{k-d+1})_{ \ge n-k} \/. \end{equation*}
Observe that $p'_*(\lambda_y(\mathcal{S}_{k-d})) =1$,
by \Cref{cor:relativeBWB}. Using this and the projection formula, we calculate the left hand side of \Cref{E:qdualfinal2}:
\[\begin{split}  & p'_*(\lambda_y(\mathcal{S}_{k-d}) \cdot \lambda_y(\C^n/\mathcal{S}_{k-d})_{> n-k}) \\ 
& = p'_*(\lambda_y(\mathcal{S}_{k-d}) \cdot \lambda_y(\C^{n}/\mathcal{S}_{k-d+1})_{>n-k}) \\ & \quad
+ p'_*(\lambda_y(\mathcal{S}_{k-d}) \cdot (\lambda_y(\mathcal{L}) -1) \cdot  \lambda_y(\C^{n}/\mathcal{S}_{k-d+1})_{ \ge n-k} )\\ &=
\lambda_y(\C^{n}/\mathcal{S}_{k-d+1})_{>n-k}) + p'_*((\lambda_y(\mathcal{S}_{k-d+1}) - \lambda_y(\mathcal{S}_{k-d}))  \cdot \lambda_y(\C^{n}/\mathcal{S}_{k-d+1})_{ \ge n-k}) \\
& = \lambda_y(\C^{n}/\mathcal{S}_{k-d+1})_{>n-k}) + (\lambda_y(\mathcal{S}_{k-d+1}) - 1) \cdot \lambda_y(\C^{n}/\mathcal{S}_{k-d+1})_{ \ge n-k} \\
& =\lambda_y(\mathcal{S}_{k-d+1}) \cdot \lambda_y(\C^{n}/\mathcal{S}_{k-d+1})_{ \ge n-k} 
- y^{n-k} \wedge^{n-k} (\C^{n}/\mathcal{S}_{k-d+1})
\/. \end{split} \]
The last expression is the right hand side of \Cref{E:qdualfinal2}, again by projection formula.\end{proof}
\begin{cor}\label{conj2:qSvee} Let $i>0$. The following equalities hold 
$\QK_T(\Gr(k;n))$: 

(a) $\wedge^i \cQ^* \star \det \cS^* = (1-q) \wedge^{n-k-i} \cQ  \cdot \det(\C^n)^*$;

(b) $\wedge^i \cS \star \det \cS^*= (1-q) \wedge^{k-i} \cS^*$;

(c) $\wedge^{n-k-i} \cQ \star \det \cS = {\wedge^{i} \cQ^*  \cdot \det(\C^n)}$;

(d) $\wedge^{n-k-i} \cQ \star \det \cQ^* = {\wedge^{i} \cQ^*}$. 
\end{cor}
\begin{proof} All are consequences of \Cref{thm:qdual}. Part (a) applies this theorem
to the dual Grassmannian $\Gr(n - k, (\C^n)^*)$; we utilize that
there is a $T$-equivariant isomorphism 
$\Gr(k;\C^n) \simeq \Gr(n - k, (\C^n)^*)$ 
under which the $T$-equivariant bundle $\cQ$ is sent to $\cS^*$. 
Part (b) follows because as elements of $\K_T(\Gr(k;n))$, $\det \cQ = \det \cS^* \cdot \det \C^n$, 
then 
$\wedge^i \cS \star \det \cS^* = (\det \C^n)^* \wedge^i \cS \star \det \cQ = (1-q) \wedge^{k-i} \cS^*$.
Part (c) follows by multiplying (a) by $\det \cS$, then using (b), and (d) follows from (c). 
\end{proof}

\section{Relations in quantum K theory} 
Recall the tautological sequence 
$0 \to \cS \to \C^n \to \cQ \to 0$ on $X=\Gr(k;n)$. The goal of this section is to prove 
the following theorem.
\begin{thm}\label{thm:QKTWhitney} The following equalities hold in $\QK_T(X)$: 
\begin{equation}\label{E:QKWhitney}\begin{split} \lambda_y(\cS) \star \lambda_y(\cQ) & =  
\lambda_y(\C^n) - q y^{n-k} \bigl((\lambda_y(\cS) -1) \otimes \det \cQ\bigr) \\ & = 
\lambda_y(\C^n) - \frac{q}{1-q} y^{n-k} (\lambda_y(\cS) -1) \star \det \cQ \end{split}
\end{equation}
\end{thm}
The second equality follows from \Cref{thm:qdual}, therefore we focus on the first. 
Our strategy is to utilize again the `quantum=classical' identity
in order to show that for any $a \in \K_T(\Gr(k;n))$, and for any $d \ge 0$,
\begin{equation}\label{E:GWeq} \langle \lambda_y(\mathcal{S}), \lambda_y(\cQ), a\rangle_d 
= \lambda_y(\C^n) \langle 1 , a \rangle_d - y^{n-k} \langle (\lambda_y(\mathcal{S}) -1) \otimes \det (\cQ), a \rangle_{d-1} \/. 
\end{equation} 
By \Cref{lemma:QKeq} this will imply the desired identities. 
The KGW invariant from the left hand side may be calculated using the following lemma. 
\begin{lemma}\label{lemma:pfGr2} The following equality holds in $\K_T(\Gr(k-d;n))$: 
\[ \begin{split} & q_* p^*(\lambda_y(\mathcal{S})) \cdot q_* p^*(\lambda_y(\cQ)) = 
 \lambda_y(\C^n) - \lambda_y(\mathcal{S}_{k-d}) \cdot \lambda_y(\C^n/\mathcal{S}_{k-d})_{> n-k} \/.\end{split} \]
\end{lemma}
\begin{proof} This follows from \Cref{cor:pfinGr}, working in the Grassmannian
$\Gr(k-d;n)$ with tautological sequence $0 \to \cS_{k-d} \to \C^n \to \C^n/\cS_{k-d} \to 0$:
\[ \begin{split} q_* p^*(\lambda_y(\mathcal{S})) \cdot q_* p^*(\lambda_y(\cQ)) & =
\lambda_y(\cS_{k-d}) \cdot \lambda_y(\C^n/\cS_{k-d})_{\le n-k} \\ & 
=\lambda_y(\cS_{k-d}) \cdot (\lambda_y(\C^n/\cS_{k-d}) -\lambda_y(\C^n/\cS_{k-d}) _{> n-k}) \\ &
=\lambda_y(\C^n) - \lambda_y(\mathcal{S}_{k-d}) \cdot \lambda_y(\C^n/\mathcal{S}_{k-d})_{> n-k}
\/. \end{split} \]\end{proof}
We are now ready to prove \Cref{thm:QKTWhitney}.
\begin{proof}[Proof of  \Cref{thm:QKTWhitney}] We need to verify \Cref{E:GWeq}, and for that, 
we utilize the `quantum= classical' statement in \Cref{thm:q=cl}. From \Cref{lemma:pfGr2}, 
the left hand side of \Cref{E:GWeq} equals:
\[ \int_{\Gr(k-d;n)} (\lambda_y(\C^n) - \lambda_y(\mathcal{S}_{k-d}) \cdot \lambda_y(\C^n/\mathcal{S}_{k-d})_{> n-k}) \cdot q_* p^*(a) \/. \]
From \Cref{lemma:pfGr}, the right hand side equals:
\[ \begin{split} & \int_{\Gr(k-d;n)} \lambda_y(\C^n) \cdot  q_* p^*(a) - \\ &
\int_{\Gr(k-d+1;n)} \bigl(\lambda_y(\mathcal{S}_{k-d+1}) \cdot \lambda_y(\C^n/\mathcal{S}_{k-d+1})_{\ge n-k}
 + y^{n-k} \wedge^{n-k} (\C^n/\mathcal{S}_{k-d+1})\bigr) \cdot  \bar{q}_* \bar{p}^*(a) \/. \end{split} \] 
Here $\bar{p}: \Fl(k-d+1,k;n) \to \Gr(k;n)$ and $\bar{q}: \Fl(k-d+1,k;n) \to \Gr(k-d+1;n)$ are the natural projections.
After cancelling the like terms, this amounts to proving the equality
\[ \begin{split} & \int_{\Gr(k-d;n)} (\lambda_y(\mathcal{S}_{k-d}) \cdot \lambda_y(\C^n/\mathcal{S}_{k-d})_{> n-k}) \cdot q_* p^*(a) = \\ &
\int_{\Gr(k-d+1;n)} \bigl(\lambda_y(\mathcal{S}_{k-d+1}) \cdot \lambda_y(\C^n/\mathcal{S}_{k-d+1})_{\ge n-k}
 - y^{n-k} \wedge^{n-k} (\C^n/\mathcal{S}_{k-d+1})\bigr) \cdot  q_* p^*(a) \/. \end{split} \]
Recall the projection $p^{k-d}: \Fl(k-d,k-d+1,k;n) \to \Gr(k;n)$. As in the proof of \Cref{thm:qdual}, by diagram chasing and projection formula, one shows the previous equality by 
demonstrating the following:
\[ \begin{split} & p^{k-d}_*(\lambda_y(\mathcal{S}_{k-d}) \cdot \lambda_y(\C^n/\mathcal{S}_{k-d})_{> n-k}) \\ &
= p^{k-d}_*\bigl(\lambda_y(\mathcal{S}_{k-d+1}) \cdot \lambda_y(\C^n/\mathcal{S}_{k-d+1})_{\ge n-k}
 - y^{n-k} \wedge^{n-k} (\C^n/\mathcal{S}_{k-d+1})\bigr) \/. \end{split} \]
This is \Cref{E:qdualfinal}, proved in \Cref{lemma:maineq}.\end{proof}


\section{The QK Whitney presentation}
\label{sect:fn-main}

In this section we prove that the relations found in \Cref{thm:QKTWhitney}
give a presentation of the algebra $\QK_T(\Gr(k;n))$. The proof strategy 
is similar to that employed in \cite{siebert.tian:on,fulton.pandh:notes} 
from quantum cohomology:
one first proves that in the classical limit these generate the full ideal of relations, 
then one uses Nakayama-type arguments to upgrade to the quantum situation.
Unlike the quantum cohomology ring, the quantum K theory is not a graded ring. 
One may still use the usual Nakayama lemma for modules over completed rings 
to prove that the same phenomenon holds.
The necessary statements are collected in the Appendix \ref{sec:filtered} below.

As usual, we consider the Grassmannian $\Gr(k;n)$ equipped with the tautological 
sequence $0 \to \cS \to \C^n \to \cQ \to 0$. Let $X=(X_1, \ldots, X_k)$ and 
$\tilde{X}=(\tilde{X}_1, \ldots , \tilde{X}_{n-k})$ 
denote formal variables. The elementary symmetric polynomials
$e_i(X)= e_i(X_1, \ldots, X_k)$ and $e_j(\tilde{X})=
e_j(\tilde{X}_1, \ldots, \tilde{X}_{n-k})$ are algebraically independent. 
Geometrically, in $\mathrm{K}_T(\Gr(k;n))$,
\[ \lambda_y(\cS) = \prod_{i=1}^k (1+ y X_i) \/; \quad \lambda_y(\cQ) = \prod_{j=1}^{n-k} (1+ y \tilde{X}_j) \/. \]

Define by $I\subset \mathrm{K}_T(pt)[e_1(X), \ldots, e_k(X), e_1(\tilde{X}), \ldots, e_{n-k}(\tilde{X})]$ the ideal determined by the Whitney relations in $\mathrm{K}_T(\Gr(k;n))$, i.e., $I$ is generated by the coefficients of the powers of $y$ in 
\begin{equation}\label{E:Whitneyrel} \prod_{i=1}^k (1+ yX_i) \prod_{i=1}^{n-k} (1+y \tilde{X}_i) = \prod_{i=1}^{n} (1+y T_i)\/. \end{equation}
A non-equivariant variant of the following  proposition appears in \cite[\S 7]{lascoux:anneau}. 
\begin{prop}\label{prop:KTpres} There is an isomorphism of $\mathrm{K}_T(pt)$-algebras
\[ \Psi: \mathrm{K}_T(pt)[e_1(X), \ldots, e_k(X), e_1(\tilde{X}), \ldots, e_{n-k}(\tilde{X})]/I \to 
\mathrm{K}_T(\Gr(k;n)) \/, \]
sending $e_i(X) \mapsto \wedge^i \cS$ and $e_j(\tilde{X}) \mapsto \wedge^j \cQ$.  
\end{prop}
\begin{proof} Denote the ring on the left by $A$. Since 
$\lambda_y(\cS) \cdot \lambda_y(\cQ) = \lambda_y(\C^n)$, 
the homomorphism $\Psi: A \to \mathrm{K}_T(\Gr(k;n))$ is well defined. 
Consider the polynomial ring 
\[ A' := \Z[T_1, \ldots, T_n][e_1(X), \ldots, e_k(X), e_1(\tilde{X}), \ldots, e_{n-k}(\tilde{X})]  \/. \]
Note that $I$ is also an ideal in $A'$. Observe that 
$A'/I$ is a free $\Z[T_1, \ldots, T_n]$-module of rank ${n \choose k}$.
(There are several proofs. For instance, (temporarily) identify $\Z[T_1, \ldots, T_n]$ 
to $H^*_T(pt)$. Then from the Whitney relations 
$c^T(\cS) \cdot c^T(\cQ)= c^T(\C^n)$ in $H^*_T(\Gr(k;n))$,
there is an isomorphism of $H^*_T(pt)$-algebras $A'/I\simeq H^*_T(\Gr(k;n))$, sending
$e_i(X) \mapsto c_i^T(\cS)$ and $e_j(\tilde{X}) \mapsto c_j^T(\cQ)$.) 
If we regard the Laurent polynomial ring 
$\mathrm{K}_T(pt)= \Z[T_1^{\pm1}, \ldots, T_n^{\pm1}]$ as a 
$\Z[T_1, \ldots, T_n]$-algebra 
(under the natural inclusion of polynomials into Laurent polynomials), 
we obtain that
\[ A = (A'/I) \otimes_{\Z[T_1, \ldots, T_n]} \mathrm{K}_T(pt) \]
is a free $\mathrm{K}_T(pt)$-module of rank ${n \choose k}$.
We utilize this to show that $\Psi$
induces an isomorphism between the associated graded rings.
To calculate $\mathrm{gr}(A)$, we make the change of variables
$z_i = 1- X_i$ ($1 \le i \le k$), $\tilde{z}_j = 1 - \tilde{X}_j$ ($1 \le j \le n- k$),
and $\zeta_s = 1- T_s$ ($1 \le s \le n$). 
Each of these variables has degree $1$.
Under this change, $A$ becomes
\[ \frac{\K_T(pt)[e_1(z), \ldots, e_k(z);e_1(\tilde{z}), \ldots, e_{n-k}(\tilde{z})]}{
\langle \sum_{i+j =\ell} e_i(z_1, \ldots, z_k) e_j(\tilde{z}_1, \ldots, \tilde{z}_{n-k}) 
- e_\ell(\zeta_1, \ldots, \zeta_n) \rangle_{1 \le \ell \le n}} \/;\]
see also \eqref{eq:ring-math} below (with $q=0$). The variables 
$z_i = 1 - X_i$ are sent to the $\K$-theoretic Chern roots of $\cS^*$, and
similarly, $\tilde{z}_j$ to the $\K$-theoretic Chern roots of $\cQ^*$. 
We deduce that the initial term of $\Psi(e_i(z))$, and of 
$\Psi(e_j(\tilde{z}))$, equals to 
$c_i^T(\cS^*)$, respectively $c_j^T(\cQ^*)$ in $H^{*}_T(\Gr(k;n))$. Thus, 
when taking the associated graded rings, $\Psi$ recovers the usual presentation
of the equivariant cohomology ring. {This implies that $\Psi$ is 
injective. 

The surjectivity follows from the theory of factorial Grothendieck polynomials.
More precisely, {from \cite[Thm.~2.1]{buch:quiver} or \cite[Thm.~1.2]{oetjen}, see also \cite{fulton.lascoux,mcnamara:factorial,ikeda.naruse:K}},
for each partition $\lambda$, 
the equivariant Schubert class $\cO_\lambda$ is a symmetric polynomial in 
the $\K$-theoretic Chern roots $1-X_1, \ldots, 1-X_k$ of $\cS$ with coefficients in 
$\K_T(pt)$. By \eqref{E:KChern} below,
this is a $\K_T(pt)$-linear combination of the (images of) $e_i(X)$.}
Then $\Psi$ is also surjective, and this finishes the proof.\end{proof}

Recall from \Cref{thm:QKTWhitney} that in $\QK_T(\Gr(k;n))$,
\begin{equation}\label{E:rel} \lambda_y(\cS) \star \lambda_y(\cQ) = \lambda_y(\C^n) - \frac{q}{1-q} y^{n-k} (\lambda_y(\cS) -1) \star \det \cQ \/. \end{equation}
Motivated by this, define the ideal 
\[ I_q \subset \mathrm{K}_T[pt][[q]][e_1(X), \ldots, e_k(X), e_1(\tilde{X}), \ldots, e_{n-k}(\tilde{X})]\] 
generated by polynomials obtained by equating the powers of $y$ in the equality:
\begin{equation}\label{E:bothQKrels-Naksec} 
\begin{split}
\prod_{i=1}^k & (1+y X_i) \times  \prod_{j=1}^{n-k} (1+y \tilde{X}_i) \\ = & \prod_{i=1}^n (1+y T_i) 
-\frac{q}{1-q} y^{n-k} \tilde{X}_1 \cdot \ldots \cdot \tilde{X}_{n-k} \bigl( \prod_{i=1}^k (1+y X_i) -1 \bigr) \/. \end{split} \end{equation}

\begin{thm}\label{thm:qkpres} There is an isomorphism of $\mathrm{K}_T[pt][[q]]$-algebras
\[ \Psi: \mathrm{K}_T[pt][[q]][e_1(X), \ldots, e_k(X), e_1(\tilde{X}), \ldots, e_{n-k}(\tilde{X})]/I_q \to \QK_T(\Gr(k;n)) \/, \]
sending $e_i(X) \mapsto \wedge^i \cS$ and $e_j(\tilde{X}) \mapsto \wedge^j \cQ$.  
\end{thm}
\begin{proof} There exists a ring homomorphism
\[ \widetilde{\Psi}:\mathrm{K}_T[pt][[q]][e_1(X), \ldots, e_k(X), e_1(\tilde{X}), \ldots, e_{n-k}(\tilde{X})]\to \QK_T(\Gr(k;n)) \/, \]
sending $e_i(X) \mapsto \wedge^i \cS$ and $e_j(\tilde{X}) \mapsto \wedge^j \cQ$. It follows from \Cref{E:rel}
that $\widetilde{\Psi}(I_q) = 0$, therefore this induces the homomorphism $\Psi$ from the claim.
{In order to prove that $\Psi$ is an isomorphism, we will use the Nakayama-type result from
\Cref{prop:Nakiso}, applied to the case when 
\[ M = \mathrm{K}_T[pt][[q]][e_1(X), \ldots, e_k(X), e_1(\tilde{X}), \ldots, e_{n-k}(\tilde{X})]/I_q \] 
is the claimed presentation, regarded as a module over the $\langle q \rangle$-adically complete ring 
$R=\mathrm{K}_T[pt][[q]]$, and the free $R$-module
$N=\QK_T(\Gr(k;n))$.
\Cref{prop:KTpres} implies that if one takes the quotient by $\langle q \rangle$, $\Psi$ 
becomes an isomorphism. 
The fact that $M$ is a finite $R$-module
follows from \cite[Ex.~7.4,~p.~203]{eisenbud:CAbook} (cf.~\Cref{rmk:Eis-fg}), applied to 
the ideal $\mathfrak{m}=\langle q \rangle$, and then noting that 
$M$ is finite over 
\[ S=\mathrm{K}_T[pt][[q]][e_1(X), \ldots, e_k(X), e_1(\tilde{X}), \ldots, e_{n-k}(\tilde{X})] \/. \]
Note that $R$ and $S$ are Noetherian rings, e.g.~by \cite[Thm.~10.26]{AM:intro},
and that $\langle q \rangle$ is included in the Jacobson radical of $R$ because $1-a q$ is invertible in $R$, for any $a \in \K_T(pt)$.\begin{footnote}{See also \cite[\S 3  and Appendix A]{GMSXZZ:whitney-math} for more details about this argument, and how it applies to more general partial flag varieties.}\end{footnote}}
\end{proof}

\section{Physics, Wilson lines, \& Jacobian ring presentations}
\label{sect:physbase}
In this section we will outline the approach giving the 
Jacobian relations \eqref{E:eqJac} derived from 
a holomorphic function $W$ called the (twisted) superpotential
in the physics literature.
These relations will be utilized in the next
section to obtain the Coulomb branch presentation of 
$\QK_T(\Gr(k;n))$. See 
e.g.~\cite{Jockers:2018sfl,Jockers:2019wjh,Jockers:2019lwe,Gu:2020zpg,Ueda:2019qhg}
for references. 

Briefly, in the special case that a space can be realized as
$V // G$ for $V$ a complex 
vector space $V$ and $G$ a reductive algebraic group,
the quantum K theory of $V//G$ arises from a three-dimensional ``supersymmetric
gauge
theory,'' 
which is ultimately defined by $G$, a representation $\rho$ defining the
$G$ action on $V$, and a matrix of real numbers $\tilde{k}$ (of the same rank as $G$),
whose values we will give momentarily.
The three-dimensional theory lives on a three-manifold, which is taken
to be $\Sigma \times S^1$ for a Riemann surface $\Sigma$, partly as a result
of which the theory can be described as a two-dimensional theory on
$\Sigma$.  Correlation functions in the two-dimensional theory
include ``Wilson lines'' (see e.g. \cite[section 2]{Closset:2016arn})
\begin{equation}
{\rm Tr}_{\rho} P \exp\left( \int_{S^1} A \right)
\end{equation}
(for $\rho$ a representation of $G$, $A$ a connection on a principal $G$ bundle,
$P$ a path-ordering symbol)
on $S^1$ over a fixed point in $\Sigma$, and those
Wilson lines correspond to K theory elements on $V//G$. More precisely, we will identify them
with Schur functors on certain vector bundles on $V//G$ associated to representations of $G$; see \Cref{rmk:schur}
below. 
Quantum K theory of spaces described as critical loci of holomorphic
functions on noncompact symplectic quotients can also be described
physically, but in this section we focus on the simpler case of
spaces that are themselves symplectic quotients of vector spaces.
Given $V//G$, the quantum K theory relations arise as derivatives of
a holomorphic function known as the superpotential and conventionally denoted $W$;
cf.~\Cref{E:Wspace} below.

\begin{remark} At least in some cases physics makes predictions for
the
quantum K theory ring of a space realized as the critical locus of
a holomorphic function on $V//G$ (such as a hypersurface, or a complete intersection), 
but the details are beyond the scope of this paper. See e.g., \cite{Gu:2020zpg} and references therein.
\end{remark}
For simplicity, we specialize to the case that
$G = \mathrm{GL}_k$
for a positive integer $k$.
Then, define $\tilde{k}_{1}$, $\tilde{k}_{2} \in {\mathbb R}$ as follows.
Write $\rho$ as a sum of irreducible representations
\begin{equation}
\rho \: = \: \rho_1 \oplus \cdots \oplus \rho_{\ell}.
\end{equation}
Define 
\begin{equation} \label{eq:CS-levels}
\tilde{k}_{1} \: = \: -\frac{1}{2} \sum_{\gamma=1}^{\ell}
\left( {\rm Cas}_1(\rho_{\gamma}) \right)^2,
\: \: \:
\tilde{k}_{2} \: = \: k \: - \: \frac{1}{2} 
\sum_{\gamma=1}^{\ell} \frac{\dim \rho_{\gamma}}{\dim GL(k)}
 {\rm Cas}_2(\rho_{\gamma}),
\end{equation}
where Cas denotes eigenvalues of Casimir operators as in e.g., 
\cite[chapter 7]{iachello06}.

\begin{remark} The choice of $\tilde{k}_1$ and $\tilde{k}_2$ from
\eqref{eq:CS-levels} generates the standard quantum K theory. However, physics considerations suggest there are other choices,
conjecturally related to the quantum K theory with level structure considered in \cite{Ruan:2018}.
See \cite{Gu:2021yek} for further physics discussions.
\end{remark}

Then, the quantum K theory ring relations are given physically as 
the critical locus equations of a
function $W$
known as the superpotential,
which is given as
\cite[equ'n (2.33)]{Closset:2016arn},
\cite[equ'n (2.1)]{Gu:2020zpg}
\begin{eqnarray}
W & = &
\frac{\tilde{k}_{2}}{2} \sum_{a=1}^{k} \left(\ln X_{a} \right)^2   \: + \:
\frac{\tilde{k}_{1} - \tilde{k}_{2}}{2 k} 
\left( \sum_{a=1}^{k} \ln X_{a} \right)^2 
\nonumber \\
& & 
\: + \:
 \left( \ln(-1)^{k-1} q \right) \sum_{a=1}^{k} \ln X_{a} 
\nonumber \\
& & \: + \:
\sum_{\alpha} \left[ {\rm Li}_2\left( \exp (\tilde{\rho}_{\alpha}, \ln X)  \right)
\: + \: \frac{1}{4} \left(  \tilde{\rho}_{\alpha}, \ln X \right)^2 \right],
\end{eqnarray}
where we use $\ln X$ to denote a vector with components
$( \ln X_a )$, where $(X_{a})$ is
a point in $({\mathfrak R} \otimes_{\mathbb Z} {\mathbb C} - \Delta)/{\mathcal W} 
\cong ({\mathbb C}^{k} - \Delta)/{\mathcal W}$
for ${\mathfrak R}$ the root lattice of $\mathrm{GL}_k$, 
\begin{equation}\label{E:Wspace}
\Delta \: = \: \coprod_{a < b} \{ X_a = X_b \} \, 
\coprod_a \{X_a = 1 \} \, \coprod_a \{X_a = 0 \},
\end{equation}
${\mathcal W} = S_k$ the Weyl group of $U(k)$,
$(,)$ denotes a natural pairing
between root and weight lattice vectors,
and $\{ \tilde{\rho}_{\alpha} \}$ are the weight vectors
of the representation $\rho$.
It can be shown that the superpotential is invariant under the action
of the Weyl group.  

For the case of a Grassmannian $\Gr(k;n)$, 
described as the GIT quotient $V // \mathrm{GL}_k$ for $V = \mathrm{Hom}(\C^k, \C^n)$, with
$\rho$ given by a sum of $n$ copies of the fundamental representation,
\begin{equation}
\tilde{k}_1 \: = \: -n/2, \: \: \: \tilde{k}_2 \: = \: k - n/2,
\end{equation}
and the superpotential specializes to
\begin{eqnarray}
W & = &
\frac{k}{2} \sum_{a=1}^k \left( \ln X_a \right)^2 \: - \: \frac{1}{2} \left(
\sum_{a=1}^k \ln X_a \right)^2
\nonumber \\
& & \: + \: \ln\left( (-1)^{k-1} q \right) \sum_{a=1}^k \ln X_a
\: + \:
n \sum_{a=1}^k {\rm Li}_2 \left( X_a \right).
\label{eq:sup-gr}
\end{eqnarray}
\begin{remark}\label{rmk:schur}
A Wilson line in representation $\phi$ of $U(k)$ 
is the Chern character of the
Schur functor $\mathfrak{S}_{\phi} {\mathcal S}$ for ${\mathcal S}$ the universal 
subbundle, where the $X_a$ are exponentials of Chern roots.  
We give a few simple examples below:
\begin{center}
\begin{tabular}{ccc}
Representation & Schur functor & Wilson line \\ \hline
$\tiny\yng(1)$ & ${\mathcal S}$ & $e_1(X)$ \\
$\tiny\yng(2)$ & ${\rm Sym}^2 {\mathcal S}$ & $h_2(X)$ \\
$\tiny\yng(1,1)$ & $\wedge^2 {\mathcal S}$ & $e_2(X)$ 
\end{tabular}
\end{center}
\end{remark}

\begin{remark} There exists an analogous superpotential $W$ whose
derivatives encode quantum K theory relations for more general $V//G$.
Write $G$
as the complexification of a compact Lie group $G'$, and decompose the
Lie algebra of $G'$, ${\mathfrak g}$, as a sum of central pieces and
simple factors.  Very schematically, there is one $q_i$ for each
central copy of Lie $U(1)$ in ${\mathfrak g}$, and the superpotential
is a Weyl-invariant function determined by the weight vectors of the
representation $\rho$ and the ``Chern-Simons levels,'' here $\tilde{k}_{1}$,
$\tilde{k}_{2}$, which are determined by slight
generalizations of (\ref{eq:CS-levels}).
\end{remark}

We claim that the quantum K theory ring is determined
by an analogue of the Jacobian ring of $W$ (involving exponentials of
derivatives rather than just derivatives).  
For the moment, we compute the relations
generated by $W$, and then later we will observe
that the resulting ring is the quantum K theory ring of 
$\Gr(k;n)$ as presented by~\cite{Gorbounov:2014}.

Returning to the Grassmannian $\Gr(k;n)$,
using the possibly obscure fact that
\begin{equation}
y \frac{\partial}{\partial y} {\rm Li}_2(y) \: = \: - \ln(1 - y),
\end{equation}
we find for each $1 \le a \le k$ that
\begin{equation} \label{eq:ideal}
\exp\left( \frac{\partial W}{\partial \ln X_a} \right) \: = \: 1
\end{equation}
implies that
\begin{equation} \label{eq:ordqkpredict}
(-1)^{k-1} q \left( X_a \right)^k \: = \:
\left( \prod_{b=1}^k X_b \right) \left(1 - X_a \right)^n.
\end{equation}
There is also an equivariant version of these identities.
Let $T_i \in \K_T(pt)$ denote equivariant parameters.  These appear in the pertinent
physical theories as exponentials of
``twisted masses.''  Concretely,
in cases with twisted masses, the superpotential~\eqref{eq:sup-gr}
for $\Gr(k;n)$ generalizes to \cite{Ueda:2019qhg}
\begin{eqnarray*}
W & = &
\frac{k}{2} \sum_{a=1}^k \left( \ln X_a \right)^2 \: - \: \frac{1}{2} \left(
\sum_{a=1}^k \ln X_a \right)^2
\nonumber \\
& & \: + \: \ln\left( (-1)^{k-1} q \right) \sum_{a=1}^k \ln X_a
\: + \:
\sum_{i=1}^n \sum_{a=1}^k {\rm Li}_2 \left( X_a T_i^{-1} \right).
\end{eqnarray*}
Simplifying
\begin{equation}\label{E:eqJac}
\exp\left( \frac{\partial W}{\partial \ln X_a} \right) \: = \: 1
\end{equation}
for each $1 \le a \le k$, we find
 \begin{equation}\label{eq:betheeq}
 (-1)^{k-1} q (X_a)^k \prod_{j=1}^{n} T_j = \left(\prod_{b=1}^k X_b\right) \cdot  \prod_{i=1}^n (T_i - X_a) \/.
 \end{equation}
In the next section we will symmetrize these relations to obtain the Coulomb
branch presentation of $\QK_T(\Gr(k;n))$. This will be proved to be isomorphic
to the $\QK$ Whitney ring from \Cref{thm:qkpres}. 
We also note that the equations~(\ref{eq:betheeq}) are the same
as the Bethe Ansatz equations from 
\cite[equ'n (4.17)]{Gorbounov:2014}. In {\em loc.cit.},
the authors utilize an approach 
based on algebraic properties of integrable
systems and of equivariant localization, 
to obtain a distinct presentation of $\QK_T(\Gr(k;n))$.
See also \S \ref{sec:GKpres} below for a comparison. 

Observe that in the non-equivariant specialization, i.e.~when
$T_i = 1$, the equations \eqref{eq:betheeq} specialize to \eqref{eq:ordqkpredict}.

\begin{remark} If one considers the specialization $q \mapsto 1$ in the 
finite difference operator
$\mathcal{H}$ annihilating the $I$-function considered by Givental and Yan in
\cite[p. 21]{Givental:2020} (see also \cite{Wen:2019}), one recovers again the vacuum (or the Bethe Ansatz)
equations from \eqref{eq:betheeq}.\begin{footnote}{More precisely, in {\em loc.~cit.} one needs to take the 
$T$-equivariant version of all objects involved, and their $q$ is the `loop parameter,' which is different from ours.}\end{footnote}
The $I$-function from \cite{Givental:2020} corresponds to the `abelianization' 
of the Grassmannian $\Gr(k;n)$; cf.~{\em loc.cit.} From this perspective, our procedure below may be 
interpreted as a symmetrization of 
the specialization $\mathcal{H}_{q \mapsto 1}$. This further
suggests that the abelian-nonabelian correspondence
proved for quantum cohomology in \cite{ciocan.kim.sabbah:abelian} may extend to 
quantum K theory; see \cite{gonzalez.woodward:qkirwan}. See 
also \cite{Iritani:2013qka} for a related method to obtain relations in quantum K theory 
by identifying difference operators which annihilate the appropriate $J$-function.
\end{remark}
\begin{example} In the case of $\Gr(2;5)$,
the superpotential is given by
\begin{eqnarray}
W & = & \frac{1}{2} \left( \ln X_1 \right)^2 \: + \:
\frac{1}{2} \left( \ln X_2 \right)^2 \: - \: \left( \ln X_1 \right)
\left( \ln X_2 \right)
\nonumber \\
& & \: + \: \ln\left( - q \right) \sum_{a=1}^2 \ln X_a
\: + \:
\sum_{i=1}^5 \sum_{a=1}^2 {\rm Li}_2 \left( X_a T_i^{-1} \right).
\end{eqnarray}
and the chiral ring relations~(\ref{eq:betheeq}) are, for $a \in \{1, 2 \}$,
\begin{equation}
\prod_{i=1}^5 \left( T_i - X_a \right) \: = \:
(-q) \frac{ X_a^2 }{ X_1 X_2} \prod_{j=1}^5 T_j.
\end{equation}
In the nonequivariant case, we take $T_i = 1$, then the chiral ring relations
become
\begin{equation}
-q X_1 \: = \: X_2 (1-X_1)^5, \: \: \:
-q X_2 \: = \: X_1 (1-X_2)^5,
\end{equation}
in agreement with~(\ref{eq:ordqkpredict}). We show in the next section
how the symmetrization of this leads the quantum $\K$ relations; see also Appendix \ref{app:g25} below.
\end{example}


\section{Coulomb branch and quantum Whitney presentations}
\label{sect:shifted}
The goal of this section is to obtain the Coulomb branch 
presentation we denoted by $\widehat{\QK}_T(\Gr(k;n))$, and predicted by physics. 
We will prove in \Cref{thm:main-result}
that this is equivalent to the
presentation of $\QK_T(\Gr(k;n))$ from \Cref{thm:qkpres}.

The idea of obtaining the Coulomb branch presentation was already used in the authors' previous
work \cite{Gu:2020zpg}. First, one rewrites the `vacuum equations' \eqref{eq:betheeq}
in terms of the `shifted variables' from \eqref{E:shiftedvars} below. Since
$\Gr(k;n) = Hom(\C^k, \C^n)//\mathrm{GL}_k$, any presentation has to satisfy a
symmetry with respect to $S_k$, the Weyl group of $\mathrm{GL}_k$.
While the ideal generated by equations \eqref{eq:betheeq}
is symmetric under permutations in $S_k$, the individual generators
are not. To rectify this, we write down a `characteristic polynomial' (cf.~\Cref{eq:char} below)
where all the coefficients
satisfy the required symmetry. Then we utilize the Vieta equations for this polynomial to obtain
a set of polynomial equations. These are $S_k \times S_{n-k}$ symmetric, and they give a presentation
of the equivariant quantum K ring. To start, define
\begin{equation}\label{E:shiftedvars}
\zeta_i \: = \: 1 - T_i, \: \: \: z_a \: = \: 1-X_a \quad (1 \le i \le n\/; \quad 1 \le a \le k) \/,
\end{equation}
so that for any $a$, equation~\eqref{eq:betheeq} becomes
\begin{equation}\label{E:BA-cov} 
\left( \prod_{i=1}^n (z_a - \zeta_i ) \right)
\left( \prod_{b \neq a} ( 1- z_b) \right)
\: + \:
(-1)^k q (1-z_a)^{k-1} \prod_{i=1}^n (1-\zeta_i) \: = \: 0 \/.
\end{equation}
A key observation is that we may rewrite this in the form
\begin{equation}  \label{eq:betheeq2}
(z_a)^n \: + \: \sum_{i=0}^{n-1} (-1)^{n-i} (z_a)^i \hat{g}_{n-i}(z,\zeta,q)
\: = \: 0 \/.
\end{equation}
where the $\hat{g}_i(z,\zeta,q)$ are {\em symmetric polynomials} in $z_1, \ldots, z_k$ and
$\zeta_1, \ldots, \zeta_n$. To do this, it suffices to take $a=1$. 
The only non-trivial part is the symmetry in
the variables $z$. This follows from
repeated application of the identity
\[ e_j(z_2, \ldots, z_k) = e_j(z_1, \ldots, z_k) - z_1 e_{j-1}(z_2, \ldots , z_k) \]
to the factor $\prod_{j=2}(1-z_j)= \sum_{i=0}^{k-1} (-1)^i e_i(z_2, \ldots, z_k)$,
then collecting the resulting powers of $z_1$.

To state the formula for the polynomial
$\hat{g}_\ell(z, \zeta,q)$, we fix some notation. 
Set 
\[ c^z= \prod_{i=1}^k (1-z_i) = \sum_{i \ge 0} (-1)^i e_i(z) \/; ~ c_{\le j}^z = \sum_{i = 0}^j (-1)^i e_i(z) \/;~ c_{\ge j}^z = (-1)^j( c(z)- c_{\le j-1}(z)) \/.\] 
Informally these are truncations of the Chern polynomial in $z=(z_1, \ldots, z_k)$.
One defines similarly $c^\zeta, c_{\le j}^\zeta,c_{\ge j}^\zeta$; these are polynomials
in $\zeta=(\zeta_1,\ldots, \zeta_n)$. Set
\begin{equation}\label{E:cplambda} c'_{\ge \ell}(z,\zeta)= e_\ell(\zeta) + e_{\ell-1}(\zeta) c_{\ge 2}^z+ e_{\ell-2}(\zeta) c_{\ge 3}^z+ \ldots + e_{\ell - k+1}(\zeta) c_{\ge k}^z \/. \end{equation}
If clear from the context, we will drop the variables $z,\zeta$ from the notation. 
(As usual, $e_i(z) =e_i(\zeta) = 0$ for $i < 0$ and 
$e_0(z) = e_0(\zeta) =1$.) Define the matrices 
\[ E= \begin{pmatrix} -1 & 0 & \ldots & 0 \\ -e_1 & -1 & \ldots & 0 \\
\vdots & \vdots & \ddots & 0 \\
-e_{k-1} & -e_{k-2} & \ldots & -1 \end{pmatrix}\/; \]
\[ C^\zeta_{\ge n-k+2}= \begin{pmatrix} c_{\ge n-k+2}^\zeta \\  \vdots \\ c_{\ge n}^\zeta \\ 0 \end{pmatrix}\/; \quad 
C^{z,\zeta}_{\ge n-k+1}= \begin{pmatrix}  c'_{\ge n-k+1} \\ c'_{\ge n-k+2} \\  \vdots \\ 
c'_{\ge n} \end{pmatrix} \/.\] 
Besides their usefulness in the lemma below, these matrices will 
appear naturally in \S \ref{sec:pres-iso} below, in relation to an
equivariant generalization of Grothendieck polynomials. 
\begin{lemma}\label{lemma:gell} The polynomial coefficients $\hat{g}_{\ell}(z,\zeta,q)$ from \eqref{eq:betheeq2} are given by:
\begin{equation}\label{E:gell}
\begin{cases} c'_{\ge \ell}(z,\zeta) & \textrm{ if } 
1 \leq \ell \leq n-k \\ c'_{\ge \ell}(z,\zeta) + \bigl(E\cdot C^\zeta_{\ge n-k+2}\bigr)_\ell 
+ (-1)^{n+k} q \binom{k-1}{n-\ell} c^\zeta 
&
\textrm{ if } 
n-k+1 \le \ell \le n \/. \end{cases}
\end{equation}
Here $M_\ell$ denotes the $\ell-n+k$-th component of the $k$-component column matrix $M$.
\end{lemma}
\begin{proof} This is a tedious, but rather standard algebraic manipulation, which we leave to the reader. \end{proof}

Define a `characteristic polynomial' $f(\xi, z, \zeta,q)$ by
\begin{equation} \label{eq:char}
f(\xi,z,\zeta,q) \: = \:
\xi^n \: + \: \sum_{i=0}^{n-1} (-1)^{n-i} \xi^i \hat{g}_{n-i}(z,\zeta,q) \/.
\end{equation}
From \eqref{eq:betheeq2} we deduce that
$f(\xi,z,\zeta,q) = 0$ whenever $\xi= z_a$ for some $a$.
Since $f$ is a degree $n$ polynomial in $\xi$,
the equation $f(\xi,z,\zeta,q) = 0$
has $n$ roots in some appropriate field extension, which by construction include
$z_1, \ldots , z_k$. Let
$\{z, \hat{z} \} = \{z_1, \cdots, z_k; \hat{z}_{k+1}, \cdots, \hat{z}_n \}$ denote the
$n$ roots of~(\ref{eq:char}). From Vieta's formula,
\begin{equation}\label{E:Vietazhatz}
\sum_{i+j = \ell} e_i(z) e_j(\hat{z}) \: = \: \hat{g}_{\ell}(z,\zeta,q) \/.
\end{equation}
This determines the `Coulomb branch ring':
\begin{equation}  \label{eq:ring-phys}
\widehat{\QK}_T(\Gr(k;n))= \K_T(pt)[[q]][e_1(z), \cdots, e_k(z), e_1(\hat{z}), \cdots, e_{n-k}(\hat{z})] /
\hat{J},
\end{equation}
where $\hat{J}$ is the ideal generated by the polynomials
\begin{equation} \label{eq:l1}
\sum_{i+j=\ell} e_i(z) e_j(\hat{z}) - \hat{g}_{\ell}(z,\zeta,q)\/; \quad 1 \leq \ell \leq n \/. 
\end{equation}
Our goal is to demonstrate that the relations above are equivalent
to those from \Cref{thm:qkpres} arising from the $\QK$-Whitney relations. 
To this aim, apply the same change of variables \eqref{E:shiftedvars} to the relations
\eqref{E:bothQKrels-Naksec}, denoting in addition $\tilde{z}_i = 1 - \tilde{X}_i$. Then
the presentation in~\Cref{thm:qkpres} can be written as:
\begin{equation} \label{eq:ring-math}
\widetilde{\QK}_T(\Gr(k;n))= \K_T(pt)[[q]][e_1(z),\ldots,e_k(z),e_1(\tilde{z}),\ldots,e_{n-k}(\tilde{z}
)]/{\tilde{J}},
\end{equation}
where the ideal $\tilde{J}$ is generated by 
$\sum_{i+j=\ell} e_{i}(z) e_{j}(\tilde{z}) \: - \: \tilde{g}_{\ell}
(z,\zeta,q)$ for
\begin{equation}\label{E:tildeg}
\tilde{g}_{\ell}
(z,\zeta,q) \: = \: e_{\ell}(\zeta)
 \: - \:
\frac{q}{1-q}\sum_{s=n-k+1}^{\ell}(-1)^s \binom{n-s}{\ell-s} \Delta_{s+k-n},
\end{equation}
for $1 \leq \ell \leq n$, and $\Delta_{i} = e_{i}(1-z) e_{n-k}(1-\tilde{z})$. 
Note that $\tilde{g}$ also depends on $\tilde{z}$, although this is not included 
in the notation. (In fact, in the proof of \Cref{thm:main-result} we will eliminate the dependence
on $\tilde{z}$; see \Cref{lemma:twopres}.) To get a more 
explicit formula, observe that
\begin{equation}\label{E:KChern} e_{i}(1-x_1, \ldots, 1-x_n) = \sum_{s=0}^i (-1)^s {n-s \choose i-s} e_s(x_1, \ldots, x_n) \/.\end{equation}
An easy algebra manipulation based on this formula shows that for $1 \le \ell \le n$,
\[ \sum_{s=n-k+1}^{\ell}(-1)^s \binom{n-s}{\ell-s} e_{s+k-n}(1-z) = (-1)^{n-k+1}\Bigl({k \choose \ell +k-n}- e_{\ell+k-n}(z)\Bigr) \/,\]
therefore \eqref{E:tildeg} may be rewritten as
\begin{equation}\label{E:tildeg2} 
\tilde{g}_{\ell}
(z,\zeta,q) \: = \: e_{\ell}(\zeta) +
 (-1)^{n-k}
\frac{q}{1-q} e_{n-k}(1-\tilde{z}) \Bigl({k \choose \ell +k-n}- e_{\ell+k-n}(z) \Bigr) \/.
\end{equation}
\begin{thm}  \label{thm:main-result}
The following three rings are isomorphic to one another and to
the algebra $\QK_T(\Gr(k;n))$:
\begin{enumerate}[label=(\alph*)]
\item \label{pres:hatx} The ring $\widehat{\QK}_T(\Gr(k;n))$ from \eqref{eq:ring-phys};

\item \label{pres:shx} The ring $\widetilde{\QK}_T(\Gr(k;n))$ from \eqref{eq:ring-math};

\item \label{pres:x} The ring
${\mathrm K}_T(pt)[[q]] [e_1(X), \ldots, e_k(X), e_1(\tilde{X}), \ldots, e_{n-k}(\tilde{X})] / I_q$,
where $I_q$ is the ideal defined in \eqref{E:bothQKrels-Naksec}. 
\end{enumerate}
\end{thm}
Note that the ring in (c) was proved in \Cref{thm:qkpres} to be 
isomorphic to the `geometric' ring $\QK_T(\Gr(k;n))$. 
We already proved that the isomorphism of the rings in 
(b) and (c) follows from the change of 
variables
\begin{equation*}
z_i \: = \: 1 - X_i \/ \quad (1 \le i \le k) \/; \quad  
\tilde{z}_j \: = \: 1 - \tilde{X}_j\/ \quad (1 \le j \le n-k) \/.
\end{equation*}
The isomorphism between (a) and (b) is proved in the next section. In the process
we will reformulate these presentations in terms of (equivariant) 
Grothendieck polynomials and (equivariant) complete 
homogeneous symmetric functions; see
\Cref{lemma:twopres} below. An 
example for $\QK_T(\Gr(2;5))$ is given in Appendix \ref{app:g25}.

\section{An isomorphism between the Whitney and Coulomb branch presentations}\label{sec:pres-iso} 
The goal of this section is to prove 
\Cref{thm:isomorphism}, thus finishing 
the proof of \Cref{thm:main-result}. We utilize the notation from \S \ref{sect:shifted}. 

\subsection{Grothendieck polynomials} We start by recording some algebraic identities
about the Grothendieck polynomials $G_j(z)=G_j(z_1, \ldots, z_k)$, indexed by single row partitions.
As usual, $z=(z_1, \ldots, z_k)$, and $e_i=e_i(z), h_i=h_i(z)$ denote the elementary symmetric function,
respectively the complete homogeneous symmetric function.
It was proved in \cite[p.80]{lenart:combinatorial} that
\begin{equation}\label{E:Gtoh}
G_j(z) = \sum_{a,b \ge 0\/; a+b \le k} (-1)^b h_{j+a} e_b = h_j+(h_{j+1} - h_j e_1) + (h_{j+2} - h_{j+1} e_1+ h_j e_2) +  \ldots  \/. 
\end{equation}
An equivalent formulation proved in
\cite[Thm. 2.2]{lenart:combinatorial}) is 
\begin{equation}\label{E:Gj} 
G_j(z) = h_j(z) - \sum_{a=2}^{k} (-1)^{a} s_{(j,1^{a-1})}(z) \/, 
\end{equation}
where $s_\mu(z)$ denotes the Schur polynomial, and $(j,1^{a-1})$ is the partition
$(j, 1, \ldots, 1)$ with $a-1$ $1$'s. We refer to \cite{lenart:combinatorial} for more about these polynomials. 
We record a Cauchy-type identity for Grothendieck polynomials. It is likely well known, but we
could not find a reference. 
\begin{lemma} If $\ell \ge 1$ and $z=(z_1, \ldots, z_k)$, then 
\begin{equation}\label{E:cauchy-groth} \sum_{i,j \ge 0\/; i+j=\ell} (-1)^j e_i(z) G_j(z) = e_{\ell+1}(z) - e_{\ell+2}(z) + \ldots  \/. \end{equation}
In particular for $\ell \ge k$, 
$\sum_{i,j \ge 0 \/; i+j=\ell} (-1)^j e_i(z) G_j(z) = 0$.
\end{lemma}
{\begin{proof}
From \eqref{E:Gj} we obtain
    \begin{align*}
      \sum_{i+j=\ell} (-1)^j e_i(z) G_j(z) &\:=\:  e_\ell(z) + \sum_{j=1}^{\ell} \sum_{a=1}^{k}(-1)^{j+a-1} e_{\ell-j} (z) s_{(j,1^{a-1})}(z) \\
        &\:=\:  e_\ell(z) + \sum_{a=1}^k (-1)^{a-1} \left[\sum_{j=1}^{\ell} (-1)^{j} e_{\ell-j}(z) s_{(j,1^{a-1})}(z) \right]\ .
    \end{align*}
To prove the lemma, it suffices to show that 
\[\sum_{j=1}^{\ell} (-1)^{j} e_{\ell-j}(z) s_{(j,1^{a-1})}(z) = - e_{\ell + a- 1}(z) \/. \]
The case when $a=1$ follows from  
$\sum_{i+j=\ell} (-1)^j e_i(z) h_j(z) =0$ (i.e., the usual Cauchy formula)
for $\ell \ge 1$. For $a>1$ we utilize the Pieri formula   
to multiply a Schur function by $e_i(z)$ (cf~e.g.~\cite{fulton:young}).
To start, only Schur functions $s_\mu$ with $\mu_1 \le \ell+1$ and $1 \le \mu_2 \le 2$ 
may appear.
Furthermore, if $\mu=(\mu_1, \ldots, \mu_s)$ is such a partition with $1<\mu_1\le \ell+1$,
then the Schur function $s_\mu$ appears in the left hand side twice, 
in the expansion of the terms
\[ (-1)^{\mu_1}e_{\ell-\mu_1}(z) s_{\mu_1,1^{a-1}}(z) + (-1)^{\mu_1-1}e_{\ell-\mu_1+1}(z) s_{\mu_1-1,1^{a-1}}(z) \/. \] 
Since the coefficients of $s_\mu$ appearing in the Pieri rule equal to $1$, and since 
the terms above have opposite signs, it follows that the corresponding 
$s_\mu$'s cancel. We are left with the situation when $\mu_1 =1$. Then necessarily 
$j=1$, i.e., we consider the terms $- e_{\ell-1}(z) e_{a}(z)$. Again by Pieri formula, 
the only multiple of $s_\mu$ with $\mu_1=1$ is $- e_{\ell+a-1}(z)$. 
This finishes the proof.\end{proof}
}

Define the column matrices:
\[ H= \begin{pmatrix} (-1)^{n-k+1} h_{n-k+1} \\ (-1)^{n-k+2} h_{n-k+2} \\ \vdots \\ (-1)^{n} h_{n} \end{pmatrix} \/; \quad 
G= \begin{pmatrix} (-1)^{n-k+1} G_{n-k+1} \\ (-1)^{n-k+2} G_{n-k+2} \\ \vdots \\ (-1)^{n} G_{n} \end{pmatrix} \/.\]
Define also the $k \times k$ matrix
\begin{equation*}\label{E:Amat}
{
A \: = \:
\begin{pmatrix}
c_{ \le k-1} &- c_{ \le k-2} &\cdots &\cdots &\cdots &\cdots         &(-1)^{k-1} \\ 
                (-1)^{k-2}c_{ \ge k}     &c_{\le k-2}           &\cdots &\cdots &\cdots &\cdots         &(-1)^{k-2} \\ 
                 \vdots  & \ddots       &\ddots &\ddots &\cdots &\cdots &\vdots \\
                 (-1)^{k-i}c_{ \ge k} &\cdots    & (-1)^{k-i} c_{ \ge k-i+2}      & c_{ \le k-i}  &-c_{\le k-i-1}   &\cdots  &(-1)^{k-i} \\
                 \vdots & \ddots        &\ddots &\ddots &\ddots &\cdots &\vdots \\
                 -c_{\ge k} &\cdots &\cdots &\cdots &\cdots &c_{\le 1} &-1 \\
                 c_{ \ge k}  &\cdots &\cdots &\cdots &\cdots  &c_{\ge 2} & 1 \\
\end{pmatrix}
}
\end{equation*}
\begin{lemma}\label{lemma:GAH} The following equality holds: $G = A H$. 
\end{lemma}
\begin{proof} This follows from \eqref{E:Gtoh} together with the Cauchy formula
$\sum_{a+b=j} (-1)^a h_a e_b =0$ for $j \ge 1$.
\end{proof}

\subsection{Equivariant deformations} Motivated by propositions \ref{prop:Ehat_groth} and \ref{prop:Etilde_h} 
below, define the polynomials
$h'_j,G'_j \in \C[z,\zeta]$ by
\[  h'_j(z,\zeta) = \sum_{a+b=j} (-1)^{a} e_a(\zeta) h_b(z)\/; \quad 
G_j' (z,\zeta)=\sum_{a+b=j} (-1)^{a} e_a(\zeta) G_b(z) \/. \]
Note that both are symmetric polynomials in variables $z$ and $\zeta$. 
The polynomial $h'_j(z,\zeta)$ is a specialization of the factorial complete 
homogeneous polynomial (see e.g.~\cite{mihalcea:giambelli}), but
the polynomial $G'_j(z,\zeta)$ is far from being the factorial 
Grothendieck polynomial \cite{mcnamara:factorial}, which is much more complicated. 
\begin{footnote}{The factorial Grothendieck polynomials, which represent the equivariant
Schubert classes in $\K_T(\Gr(k;n))$, are not symmetric in $\zeta$.}\end{footnote} 
\begin{lemma}\label{lemma:cauchy-h} The following Cauchy-type formulae hold: 

(a) If $\ell \ge 1$ then 
\[ \sum_{a+b=\ell} (-1)^a h_a'(z,\zeta) e_b(z) = e_\ell(\zeta) \/. \]

(b) If $\ell \ge 1$ then
\[ \sum_{a+b=\ell} (-1)^a G_a'(z,\zeta) e_b(z) = c'_{\ge \ell}(z,\zeta) \/,\]
where $c'_{\ge \ell}(z,\zeta)$ is defined in \eqref{E:cplambda}.
\end{lemma}
\begin{proof} For (a), we collect the coefficients of $e_i(\zeta)$ from both sides. On the left hand side this coefficient equals to 
\[ \sum_{i \le a} (-1)^{a} e_{\ell-a}(z) (-1)^{i} h_{a-i}(z) = \sum_{i \le a} (-1)^{a-i} e_{\ell-a} (z) h_{a-i}(z) \/. \]
By Cauchy formula, this equals to $0$ if $i <\ell$, and it equals to $1$ if $i=\ell$, thus proving part (a). 
We apply the same strategy for (b). The coefficient of $e_\ell(\zeta)$ on the left hand side equals to $1$, 
and a calculation similar to (a) shows that if $i < \ell$, the coefficient of $e_i(\zeta)$ is
\[ \sum_{i \le a} (-1)^{a} e_{\ell-a} (z) (-1)^i G_{a-i}(z) = e_{\ell-i+1} (z) - e_{\ell-i+2} (z) + \ldots = c_{\ge \ell - i+1}^z \/. \]
Here the first equality follows from the Cauchy-type formula 
\eqref{E:cauchy-groth}. This finishes the proof of (b).\end{proof}
Our next task is to write down the analogue of \Cref{lemma:GAH} relating the polynomials $G'_\ell$ and $h'_\ell$. To this aim,  
in analogy to \eqref{E:Gtoh}, define the polynomial
\begin{equation}\label{E:modG}
G_j^{\dagger}(z,\zeta) = \sum (-1)^b h'_{j+a}(z,\zeta) e_b(z) = h'_j+(h'_{j+1} - h'_j e_1) + (h'_{j+2} - h'_{j+1} e_1+ h'_j e_2) +  \ldots  \/, \end{equation}
where the sum is over $a,b \ge 0$ and $a+b \le k$. 

\begin{lemma}\label{lemma:G!toG'} For any $\ell \ge 1$, the following holds:
\[ G_\ell^{\dagger}(z,\zeta) - G'_\ell(z,\zeta) = (-1)^{\ell+1} (e_{\ell+1}(\zeta) - e_{\ell+2}(\zeta) + \ldots + (-1)^{k-1} e_{\ell+k}(\zeta)) \/.\]
\end{lemma}
\begin{proof} This follows from a direct calculation, by identifying the coefficients of
$e_i(\zeta)$ in both sides. We leave the details to the reader.\end{proof} 
Define the column matrices:
\[ H'= \begin{pmatrix} (-1)^{n-k+1} h'_{n-k+1} \\ (-1)^{n-k+2} h_{n-k+2} \\ \vdots \\ (-1)^{n} h'_{n} \end{pmatrix} \/; \quad 
G'= \begin{pmatrix} (-1)^{n-k+1} G'_{n-k+1} \\ (-1)^{n-k+2} G'_{n-k+2} \\ \vdots \\ (-1)^{n} G'_{n} \end{pmatrix} \/.\]

\begin{cor}\label{cor:eqGAH} The following holds: $G'=AH'+C^\zeta_{\ge n-k+2}$. 
\end{cor} 
\begin{proof} Let $G^\dagger$ be the column matrix with entries $(-1)^{n-k+i} G^\dagger_{n-k+i}(z,\zeta)$. 
As in the proof of \Cref{lemma:GAH}, $G^\dagger = A H'$. Then the claim follows from the formula in \Cref{lemma:G!toG'}.\end{proof}
\subsection{Elimination of variables} Our strategy is to utilize the Vieta relations in order to
eliminate the variables $\hat{z}_i, \tilde{z}_i$ from the quantum rings $\widehat{\QK}_T(\Gr(k;n))$
and $\widetilde{\QK}_T(\Gr(k;n))$ respectively, then prove that the resulting rings are isomorphic.
It turns out that the elimination process naturally leads to the study of variations of the 
Grothendieck, respectively
the complete homogeneous polynomials.
\begin{prop}\label{prop:Ehat_groth} 
Consider the Coulomb branch ring
$\widehat{\QK}_T(\Gr(k;n))$ from \eqref{eq:ring-phys}.
Then the following holds for $1 \le \ell \le n-k$:
\[ e_\ell(\hat{z}) \equiv \sum_{i+j=\ell} (-1)^j e_i(\zeta) G_j(z) = (-1)^\ell G'_\ell(z,\zeta) \mod \hat{J} \/. \]
\end{prop}
\begin{proof} We proceed by induction on $\ell \ge 1$. If $\ell =1$, the claim follows from 
the definition of $g_1(z,\zeta,q)$ and the fact that $G_1(z) = \sum_{i=1}^k (-1)^i e_i(z)$
from \Cref{E:Gj}. By induction and the relations \eqref{E:Vietazhatz}
we obtain that for $\ell \ge 2$,
\begin{equation}\label{E:induction} \begin{split} e_\ell(\hat{z}) & \equiv g_\ell(z,\zeta,q) - \sum_{s=0}^{\ell-1}
e_{\ell -s}(z) \left(\sum_{i+j=s} (-1)^j e_i(\zeta) G_j(z) \right) \\ & = g_\ell(z,\zeta,q) - \sum_{s=0}^{\ell-1}
e_{\ell -s}(z) (-1)^s G'_s (z,\zeta) 
\\ & =  g_\ell(z,\zeta,q) + (-1)^\ell G'_\ell(z,\zeta) - c'_{\ge \ell} (z,\zeta)
\\ & = (-1)^\ell G'_\ell(z,\zeta) 
\/. \end{split} \end{equation}
Here the third equality follows from \Cref{lemma:cauchy-h}(b), and the fourth follows from the 
definition of $\hat{g}_\ell$ from \Cref{lemma:gell}.
\end{proof}

\begin{prop}\label{prop:Etilde_h} Consider the quantum $\K$ Whitney ring
$\widetilde{\QK}_T(\Gr(k;n))$ from \eqref{eq:ring-math}. 
Then the following holds for $1 \le \ell \le n-k$:
\[ {e_\ell(\tilde{z})} \equiv \sum_{i+j=\ell} (-1)^j e_i(\zeta) h_j(z) = (-1)^\ell h'_\ell(z,\zeta) \mod \tilde{J} \/. \]
\end{prop}
\begin{proof} The proof is similar to that for \Cref{prop:Ehat_groth}, utilizing the 
Cauchy formula in \Cref{lemma:cauchy-h}(a),
and that in this case
$\tilde{g}_\ell(z,\zeta) = e_\ell(\zeta)$ by \Cref{E:tildeg2}.
\end{proof}
Recall the notation from \eqref{E:cplambda} and the surrounding paragraphs.
Define the column matrices 
\begin{equation}\label{E:rcmat} \tilde{R}= \begin{pmatrix} \tilde{g}_{n-k+1} \\ \tilde{g}_{n-k+2} \\ \vdots \\ \tilde{g}_{n} \end{pmatrix} \/; \quad  
\hat{R}= \begin{pmatrix} \hat{g}_{n-k+1} \\ \hat{g}_{n-k+2} \\ \vdots \\ \hat{g}_{n} \end{pmatrix} \/; \quad 
E_{n-k+1}^\zeta = \begin{pmatrix} e_{n-k+1}(\zeta) \\ \vdots \\ \vdots \\ e_n(\zeta) \end{pmatrix}
\end{equation}
\begin{thm}\label{lemma:twopres} (a) The ring $\widetilde{\QK}(\Gr(k;n))$
is isomorphic to $\K_T(pt)[[q]][e_1, \ldots, e_k]/\widetilde{I}$,
where the ideal $\widetilde{I}$ is defined by $\tilde{R}=E H'+ E_{n-k+1}^\zeta$, i.e.,
\[ \begin{pmatrix} \tilde{g}_{n-k+1} \\ \tilde{g}_{n-k+2} \\ \vdots \\ \tilde{g}_{n} \end{pmatrix} =
\begin{pmatrix} -1 & 0 & \ldots & 0 \\ -e_1 & -1 & \ldots & 0 \\
\vdots & \vdots & \ddots & 0 \\
-e_{k-1} & -e_{k-2} & \ldots & -1 \end{pmatrix}
\begin{pmatrix} (-1)^{n-k+1} h'_{n-k+1} \\ (-1)^{n-k+2} h'_{n-k+2} \\ \vdots \\ (-1)^{n} h'_{n} \end{pmatrix} 
+ \begin{pmatrix} e_{n-k+1}(\zeta) \\ \vdots \\ \vdots \\ e_n(\zeta) \end{pmatrix}\]

(b) The ring $\widehat{\QK}(\Gr(k;n))$ is isomorphic to 
$\K_T(pt)[[q]][e_1, \ldots, e_k]/\widehat{I}$,
where the ideal $\widehat{I}_q$ is defined by $\hat{R}=E G'+C_{\ge n-k+1}^{z,\zeta} $, i.e., 
\[ \begin{pmatrix} \hat{g}_{n-k+1} \\ \hat{g}_{n-k+2} \\ \vdots \\ \hat{g}_{n} \end{pmatrix} =
\begin{pmatrix} -1 & 0 & \ldots & 0 \\ -e_1 & -1 & \ldots & 0 \\
\vdots & \vdots & \ddots & 0 \\
-e_{k-1} & -e_{k-2} & \ldots & -1 \end{pmatrix}
\begin{pmatrix} (-1)^{n-k+1} G'_{n-k+1} \\ (-1)^{n-k+2} G'_{n-k+2} \\ \vdots \\ (-1)^{n} G'_{n} \end{pmatrix} 
+ \begin{pmatrix} c'_{\ge n-k+1} \\ \vdots \\ \vdots \\ c'_{\ge n} \end{pmatrix}
\]
\end{thm}
\begin{proof} The isomorphisms are obtained by eliminating 
the variables $\tilde{z}_i$, respectively $\hat{z}_i$ ($1 \le i \le n-k$) from the first $n-k$ relations
in $\widehat{\QK}_T(\Gr(k;n))$ respectively $\widetilde{\QK}_T(\Gr(k;n))$ 
(cf.~\eqref{eq:ring-phys} respectively \eqref{eq:ring-math}).
We indicate the main steps to obtain these formulae. 

Consider first $\widetilde{\QK}_T(\Gr(k;n))$.
By \Cref{prop:Etilde_h},
$e_j(\tilde{z})\equiv (-1)^j h'_j(z,\zeta)$ modulo $\widetilde{I}$; we utilize this to 
eliminate the first $n-k$ relations from
$\tilde{J}$ to show that $\widetilde{I}$ is generated by
\[ \sum_{a+b=\ell \/; b \le n-k} (-1)^b e_a(z) h'_b (z,\zeta)- \tilde{g}_\ell(z,\zeta) \/, \quad n-k+1 \le \ell \le n \/. \]
From the Cauchy formula in \Cref{lemma:cauchy-h}(a) it follows that for $n-k+1 \le \ell \le n$,
\[ \sum_{a+b=\ell \/; b \le n-k} (-1)^b e_a h'_b =  e_\ell(\zeta)-\sum_{b=n-k+1}^{\ell} (-1)^{b} e_{\ell - b} h'_{b} \equiv \tilde{g}_\ell \mod \widetilde{I} \/. \]
Writing these expressions in matrix form proves the claim for $\widetilde{I}$. 

Same proof works for $ \widehat{I}$, after utilizing 
that $e_j(\hat{z})\equiv (-1)^j G'_j(z,\zeta)$ modulo $\widehat{I}$ by \Cref{prop:Ehat_groth}, 
and the Cauchy-type formula from \Cref{lemma:cauchy-h}(b);
the matrix $C_{\ge n-k+1}^{z,\zeta}$ accounts for 
the right hand side of the Cauchy formula.
\end{proof}
\subsection{Proof of \Cref{thm:main-result}}
We need the following algebraic identities.
\begin{lemma}\label{lemma:ApRt} (a) Let $A'$ be the antidiagonal transpose of $A$ from \eqref{E:Amat}. Then
\[ A' \begin{pmatrix} {k \choose 1}- e_{1}(z) \\  {k \choose 2}- e_{2}(z) \\ \vdots \\ {k \choose k}- e_{k}(z) \end{pmatrix} = \prod_{i=1}^{k}(1-z_i) \begin{pmatrix} {k-1 \choose k-1} \\ {k-1 \choose k-2}\\ \vdots \\ {k-1 \choose 0} \end{pmatrix} \/. \]

(b) $\hat{R} - C_{\ge n-k+1}^{z,\zeta} - E\cdot C^\zeta_{\ge n-k+2} = (-1)^{n+k} q \prod_{i=1}^n (1-\zeta_i) 
\begin{pmatrix} {k-1 \choose k-1} \\ {k-1 \choose k-2}\\ \vdots \\ {k-1 \choose 0} \end{pmatrix} $.
\end{lemma}
\begin{proof} Part (a) follows e.g. from induction on $i$; we leave the details to the reader. Part (b) 
is equivalent to \Cref{lemma:gell}.\end{proof}

Next we state the key result which relates the two presentations.
\begin{lemma}\label{lemma:key} The following equality holds: 
\[ \hat{R} - C_{\ge n-k+1}^{z,\zeta} - E\cdot C^\zeta_{\ge n-k+2} \equiv A' (\tilde{R} - E_{n-k+1}^\zeta) \mod \widetilde{I} \/, \]
where $A'$ is the antidiagonal transpose of $A$. 
\end{lemma}
\begin{proof} From part (b) of \Cref{lemma:ApRt} we need to show that 
\[ A' (\tilde{R} - E_{n-k+1}^\zeta) \equiv (-1)^{n+k} q \prod_{i=1}^n (1-\zeta_i) \begin{pmatrix} {k-1 \choose k-1} \\ {k-1 \choose k-2}\\ \vdots \\ {k-1 \choose 0} \end{pmatrix} \/. \]
From the definition of $\tilde{g}_\ell$ from \eqref{E:tildeg2} it follows that
for $\ell \ge n-k+1$,
\[ \tilde{g}_\ell - e_{\ell}(\zeta) \equiv (-1)^{n-k}
\frac{q}{1-q} e_{n-k}(1-\tilde{z}) \Bigl({k \choose \ell +k-n}- e_{\ell+k-n}(z) \Bigr) \/. \]
By \Cref{lemma:ApRt}(a), for $n-k+1 \le \ell \le n$,
\[ \begin{split} A' (\tilde{R} - E_{n-k+1}^\zeta) & \equiv 
(-1)^{n-k} \frac{q}{1-q} e_{n-k}(1-\tilde{z}) A' \cdot \Bigl({k \choose \ell +k-n}- e_{\ell+k-n}(z) \Bigr)_{\ell+k-n} 
\\ & = (-1)^{n-k} \frac{q}{1-q} e_{n-k}(1-\tilde{z}) e_k (1-z) \begin{pmatrix} {k-1 \choose k-1} \\ {k-1 \choose k-2}\\ \vdots \\ {k-1 \choose 0} \end{pmatrix} \/. \end{split} \]
If one writes $e_{n-k}(1-\tilde{z}) e_k (1-z)$ in the presentation from \Cref{thm:main-result}(c), it follows that in 
$\widetilde{QK}_T(\Gr(k;n))$
\[ e_{n-k}(1-\tilde{z}) e_k (1-z) = e_{n-k}(\tilde{X}) e_k(X)\equiv (1-q) \prod_{i=1}^n (1- \zeta_i) \mod \widetilde{I} \quad \/. \]
Then the claim follows by combining the previous equalities.\end{proof}

\begin{cor}\label{cor:inclusion} The ideal $\widehat{I} \subset \widetilde{I}$.\end{cor}
\begin{proof} We need to show that $\hat{R}-EG'-C_{\ge n-k+1}^{z,\zeta} \equiv 0$ modulo $\widetilde{I}$.
From definitions, \Cref{lemma:key} and \Cref{cor:eqGAH} we have
\[ \begin{split} \hat{R}-EG'-C_{\ge n-k+1}^{z,\zeta} & =  \hat{R}-E(AH'+C_{\ge n-k+2}^\zeta)-C_{\ge n-k+1}^{z,\zeta}
\\ & = \hat{R} - E C_{\ge n-k+2}^\zeta - C_{\ge n-k+1}^{z,\zeta} - EAH' \\ & 
\equiv A' (\tilde{R} - E_{n-k+1}^\zeta) - EAH' \\ & \equiv A E H' - EAH' \/. \end{split} \]
Then the claim follows because $EA = A' E$ as matrices with polynomial coefficients.\end{proof}
The next theorem finishes the proof of \Cref{thm:main-result}.
\begin{thm}\label{thm:isomorphism} There is $\K_T(pt)[[q]]$-algebra isomorphism
\[ \K_T(pt)[[q]][e_1, \ldots, e_k]/\widehat{I} \to \K_T(pt)[[q]][e_1, \ldots, e_k]/\widetilde{I} \]
sending $e_i(z) \mapsto e_i(z)$. 
\end{thm}
\begin{proof}
By \Cref{cor:inclusion}, there is a surjective ring homomorphism of $\K_T(pt)[[q]]$-algebras
\[ \Phi: \K_T(pt)[[q]][e_1(z), \ldots, e_k(z)]/\widehat{I} \to \K_T(pt)[[q]][e_1(z), \ldots, e_k(z)]/\widetilde{I} \/,\]
sending $e_i(z) \mapsto e_i(z)$. To prove $\Phi$ is an isomorphism, 
we follow the same approach as in the proof of 
\Cref{thm:qkpres}, relying on \Cref{prop:Nakiso}. 
{We need to check that the Coulomb branch presentation 
$\K_T(pt)[[q]][e_1, \ldots, e_k]/\widehat{I}$ is a finite module over $\K_T(pt)[[q]]$, and that we 
obtain an isomorphism after taking the quotient by $\langle q \rangle$. To start, note that from the definition
of the ideal $\hat{J}$ from \eqref{eq:l1}, the variables
$z_1, \ldots , z_k$ are solutions of the Bethe Ansatz equations \eqref{E:BA-cov}. Then the
same happens for the variables $z_1, \ldots, z_k$ appearing in $\K_T(pt)[[q]][e_1, \ldots, e_k]/\widehat{I}$.
From this, and from the special case $q=0$ of \cite[Lemma 4.7]{Gorbounov:2014}, we obtain that 
the quotient
$\K_T(pt)[[q]][e_1, \ldots, e_k]/(\langle q \rangle + \widehat{I})$ is 
$\K_T(pt)$-module-generated by the factorial
Grothendieck polynomials $G_\lambda(z;\xi)$ from \cite{mcnamara:factorial}, where 
$\lambda$ is included in the $k \times (n-k)$ rectangle. 
(This result uses a determinantal formula for factorial Grothendieck polynomials from 
\cite[Prop. 2.14]{Gorbounov:2014}, in turn attributed to \cite{ikeda.naruse:K}.) 
Since $\Phi$ is surjective, \Cref{prop:KTpres} and \Cref{lemma:basis} imply that the images 
$\Phi(G_\lambda(z;\xi))$ modulo $q$ form a basis in 
$\K_T(\Gr(k;n))\simeq \K_T(pt)[[q]][e_1, \ldots, e_k]/(\langle q \rangle + \widetilde{I})$. 
This implies that $\Phi$ must be an isomorphism modulo $\langle q \rangle$.
Then the same argument as in \Cref{thm:qkpres}, based on \cite[Ex.~7.4,~p.~203]{eisenbud:CAbook}
(cf.~\Cref{rmk:Eis-fg}), implies that $\K_T(pt)[[q]][e_1, \ldots, e_k]/\widehat{I}$ 
is a finite module, thus $\Phi$ is an isomorphism, by \Cref{prop:Nakiso}.}
\end{proof}

\section{Examples: Non-equivariant and quantum cohomology specializations}\label{sec:2dlimit}
In this section we illustrate the non-equivariant specializations of the Coulomb $\QK$ Whitney presentations 
of $\QK_T(\Gr(k;n))$, i.e. when $\zeta_i=0$ for $1 \le i \le n$. 
The second part is dedicated to the quantum cohomology limit. In particular
we show that both presentations specialize to Witten's presentation of $\QH^*_T(\Gr(k;n))$.

\subsection{Non-equivariant presentations} We abuse notation and denote by the same symbols the 
non-equivariant specializations of relations,
ideals, etc. For $\ell \ge 1$, the polynomials $\tilde{g}_\ell, \hat{g}_\ell$ are given by
\[ \tilde{g}_\ell(z,q) = (-1)^{n-k} \frac{q}{1-q} e_{n-k}(1-\tilde{z})\Bigl({k \choose \ell +k-n}- e_{\ell+k-n}(z) \Bigr)  \/; \]
\[ \hat{g}_\ell(z,q) = c_{\ge \ell +1}^z + (-1)^{n+k} q {k-1 \choose n- \ell} \/. \]
In particular, if $1 \le \ell \le n-k$, $\tilde{g}_\ell=0$ and $\hat{g}_\ell$ does not depend on $q$.
The non-equivariant versions of propositions \ref{prop:Etilde_h} and \ref{prop:Ehat_groth} state
that for any $1 \le j \le n-k$, the following identities hold: 
\begin{equation}\label{E:cove} e_j(\tilde{z}) = (-1)^j h_j(z) \mod \tilde{J} \/; \quad e_j(\hat{z}) = (-1)^j G_j(z) \mod \hat{J} \/. \end{equation}

\begin{remark}\label{rmk:ejhatz} It is well known that the Grothendieck polynomials
$G_j(z)$ represent Schubert classes $\cO_j \in K(\Gr(k;n))$ \cite{lascoux:anneau}. 
Nevertheless, the geometric interpretation $e_j(\hat{z}) = (-1)^j \cO_j$ 
fails in the equivariant case, already for $j=1$. The Schubert divisor class may be calculated from 
the exact sequence:
\[ 0 \to \wedge^k \cS \otimes \C_{-t_1 - \ldots - t_k} \to \cO_{\Gr(k;n)} \to \cO_1 \to 0 \/. \]
Then, from the geometric interpretation of the variables $z_i,\zeta_i$,
\begin{equation}\label{E:KTO1} \cO_1 = 1 - \frac{(1-z_1) \cdot \ldots \cdot (1-z_k)}
{(1-\zeta_1) \cdot \ldots \cdot (1-\zeta_k)} \quad \in \K_T(\Gr(k;n)) \/.\end{equation}
However, from \Cref{prop:Ehat_groth},
\[ e_1(\hat{z}) = - G_1'(z,\zeta) = - G_1(z) + e_1(\zeta) \/, \]
which is different from the expression in \Cref{E:KTO1}.
\end{remark}

\begin{cor}\label{lemma:noneq-twopres} The rings 
$\widetilde{\QK}(\Gr(k;n))$ and $\widehat{\QK}(\Gr(k;n))$
are isomorphic to
\[ \widetilde{\QK}(\Gr(k;n)) = \Z[[q]][e_1, \ldots, e_k]/\widetilde{I} \/; \quad 
\widehat{\QK}(\Gr(k;n)) = \Z[[q]][e_1, \ldots, e_k]/\widehat{I} \/, \]
where the ideal $\widetilde{I}$ is defined by $\tilde{R}=E H$, i.e.,
\[ \begin{pmatrix} \tilde{g}_{n-k+1} \\ \tilde{g}_{n-k+2} \\ \vdots \\ \tilde{g}_{n} \end{pmatrix} =
\begin{pmatrix} -1 & 0 & \ldots & 0 \\ -e_1 & -1 & \ldots & 0 \\
\vdots & \vdots & \ddots & 0 \\
-e_{k-1} & -e_{k-2} & \ldots & -1 \end{pmatrix}
\begin{pmatrix} (-1)^{n-k+1} h_{n-k+1} \\ (-1)^{n-k+2} h_{n-k+2} \\ \vdots \\ (-1)^{n} h_{n} \end{pmatrix} 
\]
and the ideal $\widehat{I}$ is defined by $\hat{R}=E G+C^z_{\ge n-k+2}$, i.e., 
\[ \begin{pmatrix} \hat{g}_{n-k+1} \\ \hat{g}_{n-k+2} \\ \vdots \\ \hat{g}_{n} \end{pmatrix} =
\begin{pmatrix} -1 & 0 & \ldots & 0 \\ -e_1 & -1 & \ldots & 0 \\
\vdots & \vdots & \ddots & 0 \\
-e_{k-1} & -e_{k-2} & \ldots & -1 \end{pmatrix}
\begin{pmatrix} (-1)^{n-k+1} G_{n-k+1} \\ (-1)^{n-k+2} G_{n-k+2} \\ \vdots \\ (-1)^{n} G_{n} \end{pmatrix} 
+ \begin{pmatrix} c_{\ge n-k+2}^z \\ \vdots \\ c_{\ge n}^z\\ 0 \end{pmatrix}
\]
\end{cor}
\subsubsection{Projective spaces}
To illustrate both presentations we take two `opposite' examples, for the projective
space $\Gr(1;n)$ and its dual $\Gr(n-1;n)$. These are isomorphic manifolds, and their quantum
$\K$ rings are also isomorphic. But the isomorphism is highly non-trivial;
this is expected, given that the Grassmannians are realized as different GIT quotients. 
A worked out example for $\QK_T(\Gr(2;5)$ is included in Appendix \ref{app:g25}.

If $k=1$, then $G_n(z) = h_n(z) = z_1^n$. The presentations 
from \Cref{lemma:noneq-twopres} are
\[ \widehat{\QK}(\Gr(1;n)) = \Z[[q]][z]/\langle z^{n} - q \rangle = \widetilde{\QK}(\Gr(1;n)) \/. \]
Consider now $k=n-1$. From \eqref{E:Gj} it follows that 
the Grothendieck polynomials $G_j$ are
significantly more complicated than the polynomials $h_j$. To illustrate, consider
$\QK(\Gr(3;4))$. The $\QK$ Whitney presentation is
\[ \widetilde{\QK}(\Gr(3;4)) = \frac{\Z[[q]][e_1(z_1,z_2,z_3),e_2(z_1,z_2,z_3),e_3(z_1,z_2,z_3)]}{\langle \tilde{g}_2 + h_2, \tilde{g}_3 + e_1 h_2 - h_3, 
\tilde{g}_4 + e_2 h_2 - e_1 h_3 + h_4 \rangle } \]
where
\[ \tilde{g}_2 = \frac{q}{q-1}(1+e_1) (3-e_1) \/; \quad  \tilde{g}_3 = \frac{q}{q-1}(1+e_1)(3-e_2)\/; \quad 
\tilde{g}_4 = \frac{q}{q-1}(1+e_1)(1-e_3) \/. \]
The Coulomb presentation is 
\[ \widehat{\QK}(\Gr(3;4)) = \frac{\Z[[q]][e_1(z_1,z_2,z_3),e_2(z_1,z_2,z_3),e_3(z_1,z_2,z_3)]}{\langle \hat{g}_2 + G_2, \hat{g}_3 + e_1 G_2 - G_3- e_3, 
\hat{g}_4 + e_2 G_2 - e_1 G_3 + G_4 \rangle } \]
where $\hat{g}_2 = e_3 -  q, {\hat{g}_3} = -  2q, \hat{g}_4 = -  q$ and
 \[ \begin{split} G_2 = h_2 - s_{2,1}+ s_{2,1,1} \/; \quad G_3 = h_3 - s_{3,1}+ s_{3,1,1} \/;
\quad G_4  = h_4 - s_{4,1}+ s_{4,1,1} \/. \end{split} \]
%
\subsubsection{Relation to the Gorbounov and Korff presentation}\label{sec:GKpres}
In this section we recall the description of the non-equivariant quantum 
$\K$-theory ring $\QK(\Gr(k;n))$ from Gorbounov and Korff's paper 
\cite{Gorbounov:2014}, based on the Bethe Ansatz equations. 
Then we indicate an isomorphism between this and our 
non-equivariant Coulomb presentation $\widehat{\QK}(\Gr(k;n))$.

For a variable $\xi$, denote by $\ominus \xi= \frac{-\xi}{1-\xi}$. Consider a sequence of
indeterminates $E_1, \ldots , E_k, H_1, \ldots, H_{n-k}$, and extend this 
sequence by requiring
$E_0 = H_0 = 1$ and $H_{r} =0$ for $r > n-k$ and $E_{r}=0$ for $r > k$.
Consider the generating series
\[ H(\xi) = \sum_{r=0}^{n-k}(H_r - H_{r+1}) \xi^{n-k-r}  \/; 
\quad E(\xi)= \sum_{r=0}^{k}(E_r - E_{r+1}) \xi^{k-r} \/. \]
The following is stated in \cite[Theorem 1.1]{Gorbounov:2014}.
\begin{thm} [Gorbounov-Korff]
The non-equivariant quantum K theory ring of $\Gr(k;n)$ is generated by $H_1, \ldots, H_{n-k}$,
$E_1, \ldots, E_k$ with relations given by the coefficients of $\xi$ in the expansion of
        \begin{equation}\label{eq:GK1}
                H(\xi)E(\ominus \xi)\:=\: \left( \prod_{i=1}^k \ominus \xi \right) \xi^{n-k} (1-H_1) + q \/.
        \end{equation}
\end{thm}
Denote by $\QK^{GK}(\Gr(k;n))$ the Gorbounov and Korff's presentation.
A direct computation gives that
the coefficient of $\xi^\ell$ in \eqref{eq:GK1} is equal to 
\begin{multline}\label{eq:GK2}
                (-1)^k(H_{\ell} - H_{\ell+1}) + \sum_{j=1}^\ell (H_{\ell-j} - H_{\ell+1-j})\left[(-1)^{k-j}\sum_{s=j}^k\binom{s-1}{s-j}E_s\right] \\
                \:=\: \left\{\begin{array}{cl}
                        0 & 1\leq \ell \leq n-k-1, \\
                        (-1)^{n-\ell}q \binom{k}{n-\ell} & n-k\leq \ell \leq n \/.
                        \end{array}\right.
        \end{multline}
We provide next an algebra isomorphism
between the (non-equivariant) 
Coulomb branch presentation $\widehat{\QK}(\Gr(k;n))$
and the presentation $\QK^{GK}(\Gr(k;n))$. 
\begin{prop} Consider the map 
$\Xi: \QK^{GK}(\Gr(k;n)) \to \widehat{\QK}(\Gr(k;n))$
defined by 
\[ \Xi (H_j) = G_j(z) \textrm{ for } 1 \le j \le n-k \/; \quad \Xi(E_j) = G_{1^j}(z) \textrm{ for } 1 \le j \le k \/. \]
Then $\Xi$ is an algebra isomorphism.
\end{prop}
\begin{proof} The main argument is showing that $\Xi$ is well defined.
To calculate the image of the left hand side of \eqref{eq:GK2}, 
one utilizes the polynomial identity
$e_\ell(z) =  \sum_{j=\ell}^k\binom{j-1}{j-\ell}G_{1^j}(z)$
(see, e.g., \cite[Thm. 2.2]{lenart:combinatorial}),
together with the (non-equivariant) Cauchy identity from 
\Cref{lemma:cauchy-h}(b).~We leave the details to the reader.
\end{proof}

\subsection{Quantum cohomology as a limit}
\label{app:2dlimit} Next we discuss how the quantum K theory reduces to the quantum cohomology.
Mathematically, this is achieved by taking leading terms, or, equivalently,
taking the associated graded rings.
In physics this corresponds to `taking a two-dimensional
limit' which we recall next, see \cite{Gu:2020zpg}. 
If we assume the theory is defined on a 3-manifold
$S^1 \times \Sigma$ for some Riemann surface $\Sigma$, where
$S^1$ has diameter $L$, then 
\begin{equation}\label{E:limits}
\begin{split}
	q &= L^n q_{2d} \/, \quad X_a = \exp\left(  L \sigma_a \right) = 
	1 + L \sigma_a + \frac{L^2}{2} {\sigma_a}^2 + \cdots, \\
	T_i &= \exp\left(  L m_i \right) = 1 + L m_i + \frac{L^2}{2} {m_i}^2 + \cdots \end{split}
\end{equation}
Here $q_{2d}$ is the quantum parameter in the $2d$ theory and $m_i$ 
are elements in $H^*_T(pt)$ (in physics terminology, the twisted masses).
The term $-L m_i$ corresponds to the cohomological parameter $t_i \in H^*_T(pt)$. 
We have 
\[ z_a = 1 - X_a = -L \sigma_a - \frac{L^2}{2} {\sigma_a}^2 - \ldots \/; \quad
\zeta_i =  1-T_i = -L m_i - \frac{L^2}{2} {m_i}^2 - \ldots \/. \]

In the two-dimensional limit, $L \rightarrow 0$ and the leading terms 
will dominate. We now take the $2d$ limit in Equation~(\ref{eq:betheeq})
arising from the Coulomb branch. After taking the leading terms,
\eqref{eq:betheeq} becomes
\begin{equation*}
	\prod_{i=1}^n\left(\sigma_a - m_i \right) = (-1)^{k-1} q_{2d}, \quad a = 1, \dots, k.
\end{equation*}
These are the generators which determine the chiral rings of the $2d$ gauge linear sigma 
model (GLSM) for $\Gr(k;n)$ when twisted masses $m_i$ are turned on. 
The $2d$ limits of the polynomials $\hat{g}_{\ell}(z,\zeta,q)$
are 
$\hat{g}_{\ell}(z,\zeta,q) \rightarrow L^\ell g_{\ell}^{2d}(\sigma,m,q_{2d})$,
with $\hat{g}_{\ell}^{2d}(z,m,q_{2d})$ given by
\begin{equation*}
	\hat{g}_{\ell}^{2d}(\sigma,m,q_{2d}) = e_{\ell}(m)+ (-1)^{n+k} q_{2d} \delta_{\ell, n},\quad \ell = 1,\ldots,n.
\end{equation*}
where $\delta_{\ell, n}$ is the Kronecker delta. Then the $2d$ limit of the Coulomb branch presentation
in \eqref{eq:ring-phys} is
\begin{equation}
\label{eq:2dphysring}
	\sum_{i+j=\ell} e_{i}(\sigma) e_{j}(\hat{\sigma}) = \hat{g}_{\ell}^{2d}(\sigma,m,q_{2d}),\quad \ell=1,\ldots,n \/.
\end{equation}
We may identify $\{-\sigma_a\}$ with Chern roots of $\cS$ and $\{-\hat{\sigma}_a\}$ with Chern 
roots of $\cQ$. This recovers the equivariant quantum Whitney relations from
\eqref{E:wittenrel}:
\begin{equation*}
	c^T(\cS) \star c^T(\cQ) = c^T(\C^n) + (-1)^{k} q_{2d}.
\end{equation*}
As we observed earlier, 
the $2d$ limits of the Coulomb branch and the Whitney presentations 
from \eqref{eq:ring-phys} and \eqref{eq:ring-math} coincide. Indeed, under the change of variables from \eqref{E:limits}
and after taking leading terms we obtain
\begin{equation*}
	\tilde{g}_{\ell}^{2d}(\sigma,m,q_{2d}) = e_{\ell}(m) + (-1)^{n-k} q_{2d} \delta_{\ell, n} = {g}_{\ell}^{2d}(\sigma,m,q_{2d}), \quad \ell = 1,\ldots,n.
\end{equation*}

\appendix


\section{Completions of filtered modules}\label{sec:filtered} For the convenience of the reader, in this appendix 
we gather some results about $I$-adic completions of modules.
We utilize them to deduce properties of the (equivariant) 
quantum K theory in terms of those of the ordinary (equivariant) K theory, but 
they may be of more general interest. 
We follow the terminology from \cite[Ch.~10]{AM:intro}.

From now on we will consider a commutative ring $R$ with a descending filtration 
by additive groups 
$R = R_0 \supset R_1 \supset \ldots$ such that  
$R_i \cdot R_j \subset R_{i+j}$. 
Consider a
filtered $R$-module $M$, i.e., $M$ is equipped with a filtration by additive groups 
$M = M_0 \supset M_1 \supset \ldots$ such that $R_i M_j \subset M_{i+j}$. 

Fix $I \subset R$ an ideal. This determines filtrations
$R_n:=I^n R$ and $M_n:=I^n M$ of $R$ respectively $M$. We 
denote by $\widehat{R}, \widehat{M}$ the $I$-adic completions.

If we assume in addition that $R$ is Noetherian 
then $\widehat{R}$ is Noetherian and it is equipped with an ideal 
$\widehat{I}\simeq I \otimes_R \widehat{R}$ included in the Jacobson radical of $\widehat{R}$ \cite[Prop.~10.15]{AM:intro}. For a finitely generated $R$-module $M$, 
there is an isomorphism 
\begin{equation} \label{E:completion} M \otimes_R \widehat{R} \simeq \widehat{M} \/; \end{equation} 
see \cite[Prop.~10.13]{AM:intro}. 
Furthermore, $\widehat{R}$ is a flat $R$-algebra, in particular the operation
$M \mapsto \widehat{M}$ preserves short exact sequences. Finally, if we 
assume that $M$ is equipped with another filtration compatible with $I$, that is, 
$I. M_n \subset  M_{n+1}$, then
\eqref{E:completion}
equips $\widehat{M}$ with a natural filtration induced from $M$ which is compatible with respect to $\widehat{I}$. 

We recall next a version of Nakayama's Lemma. 
\begin{lemma}\label{cor:gen} Let $R$ be a Noetherian ring, 
$I \subset R$ an ideal, 
and let $M$ be a finitely generated $\widehat{R}$-module.
Let $M' \subset M$ be a submodule such that 
\[ {M/M'}=\widehat{I}.({M/M'}) \/.\]
Then ${M}={M'}$ as $\widehat{R}$-modules. 
\end{lemma}
\begin{proof} The claim follows from the ordinary
Nakayama Lemma \cite[Prop.~2.6]{AM:intro}, utilizing that $\widehat{I}$ is included in the Jacobson radical
of $\widehat{R}$. 
\end{proof}

\begin{lemma}\label{lemma:basis} Let $M$ be a free $R$-module of rank $p$ where $R$ is an integral domain. 
Then any set of $R$-module generators $m_1, ..., m_p \in M$ of cardinality $p$ forms a basis.\end{lemma}
\begin{proof} Let $K$ be the fraction field of $R$. Then $M \otimes_R K$ is a $K$-vector space of dimension $p$. 
Consider the map $\Phi: K^p \to M$ given by $\Phi(e_i) = m_i$, where $e_i$ is the natural basis for $K^p$. 
We have an exact sequence 
\[ 0 \to \ker(\Phi) \to K^p \to M \to 0 \/. \] 
Now $K$ is a flat $R$-module (see \cite[Cor.~3.6]{AM:intro}), and therefore we have an exact sequence of $K$-vector spaces 
\[ 0 \to \ker(\Phi) \otimes_R K \to K^p \otimes_R K \to M \otimes_R K \to 0 \/. \]
Since $\dim_K M \otimes_R K = p$, the last map is an isomorphism. The hypothesis implies that $m_1 \otimes 1, ..., m_p \otimes 1$ is a generating set for the $K$-vector space, 
therefore it is also a basis. If $c_1 m_1 + ... + c_p m_p = 0$ in $M$ then 
$c_1 m_1 \otimes 1 + ... + c_p m_p \otimes 1 = 0$, 
from where we get that $c_i = 0$, which finishes the proof. 
\end{proof}
\begin{prop}\label{prop:Nakiso} Let $R$ be a Noetherian integral domain, 
and let $I \subset R$ be an ideal.~Assume that $R$ is complete in the $I$-adic topology.
Let $M,N$ be finitely generated $R$-modules.

Assume that the $R$-module $N$, and the $R/I$-module $N/IN$, 
are both free modules 
of the same rank $p< \infty$, and
that we are given  
an $R$-module homomorphism $f: M \to N$ such that
the induced $R/I$-module map 
$\overline{f}: M/I M \to N / IN$
is an isomorphism of $R/I$-modules.

Then $f$ is an isomorphism.  
\end{prop}
\begin{proof} We start by observing that since $R$ is complete in the $I$-adic topology,
$\widehat{R} \simeq R$ and $\widehat{I} \simeq I$; cf.~\cite[Prop.~10.5 and Prop.~10.15]{AM:intro}.
Let $m_1, \ldots, m_p \in M$ be any lifts of a basis of the free $R/I$-module
$M/IM \simeq N/IN$. By \Cref{cor:gen} applied to the modules $M$ and $M' = \langle m_1, \ldots, m_p \rangle$, the elements $m_1, \ldots, m_p$ 
generate $M$, and similarly the elements $f(m_1), \ldots, f(m_p)$ 
generate $N$. Since $N$ is free of rank $p$, it follows from \Cref{lemma:basis}
that $f(m_1), \ldots, f(m_p)$ is a basis over $R$. In particular, $f$ must be surjective. 
To prove that $f$ is an isomorphism, it suffices to show that 
$m_1, \ldots, m_p$ are linearly independent over $R$. 
If $\sum c_i m_i = 0$ with $c_i \in R$, taking the image under $f$ and using
that $\{f(m_i)\}$ form a basis implies that $c_i = 0$. This finishes the proof. 
\end{proof}

{\begin{remark}\label{rmk:Eis-fg} A useful criterion for finite generation of a module 
is provided in \cite[Ex.~7.4]{eisenbud:CAbook}.
Assume that $R \subset S$ are Noetherian rings such that
$R$ is complete with respect to an ideal $\mathfrak{m} \subset S$, and that
$\mathfrak{m}$ is contained in the Jacobson radical of $S$.
If $M$ is a finitely generated $S$-module and $M/\mathfrak{m}M$ is a finitely generated
$R/\mathfrak{m}$-module, then $M$ is a finitely generated $R$-module.
See also \cite[Appendix A]{GMSXZZ:whitney-math} for more details.
\end{remark}}

\begin{remark}\label{rmk:filtration} By \cite[Prop.~10.24]{AM:intro} the hypothesis that $M,N$ are finitely generated 
may be deduced if we know that $R$ is $I$-complete, and $M,N$ are equipped with good and separated filtrations compatible with respect to $I$.~(Recall that a filtration of $M$ is good if the associated graded $\mathrm{gr} M = \bigoplus_i M_i/M_{i+1}$ 
is a finitely generated $\mathrm{gr} R$-module, and it is separated if $\bigcap_i M_i = 0$.)\end{remark}

\begin{remark} If $f: M \to N$ is a filtered homomorphism of $R$-modules (i.e. $f(M_i) \subset N_i$)
which induces an isomorphism $\mathrm{gr} f: \mathrm{gr} M \to \mathrm{gr} N$ of graded
$\mathrm{gr} R$-modules, then the induced map between the completions 
$\widehat{f}: \widehat{M} \to \widehat{N}$ is an isomorphism \cite[Lemma 10.23]{AM:intro}.
This is an alternative to \Cref{prop:Nakiso}.
\end{remark}
\begin{remark}\label{rmk:1+q} Instead of working with completions in \Cref{cor:gen} and \Cref{prop:Nakiso}, 
one may
work with the localizations with respect to the multiplicative set $S= 1 + I $. The hypothesis that
$R$ is Noetherian is not needed in this case, and
the ideal $S^{-1}I$
is included in the Jacobson radical of $S^{-1} R$; cf.~\cite[Ch.~3, Ex.~2]{AM:intro}.
If $R$ is Noetherian, then $S^{-1} R \hookrightarrow \widehat{R}$ is a subring \cite[Remark, p.~110]{AM:intro}.
\end{remark} 


\section{Example: Equivariant quantum K theory of $\Gr(2;5)$}
\label{app:g25}
In this section, we illustrate the Whitney and Coulomb branch 
presentations for $\QK_T(\Gr(2;5))$. The $\QK$ Whitney 
relations are obtained by equating powers of $y$ in 
\begin{equation}
\label{eq:ex25lambda}
\begin{split} (1+yX_1)(1+yX_2)& (1+y\tilde{X}_1)(1+y\tilde{X}_2)(1+y\tilde{X}_3)= \\ 
&\prod_{i=1}^5 (1+T_i y) - \frac{q}{1-q} y^{4}
	\tilde{X}_1 \tilde{X}_2 \tilde{X}_3 (X_1+X_2 + y X_1 X_2) \/.\end{split} 
\end{equation}
Under the changes of variable $X_i = 1-z_i$, ($1 \le i \le 2$), $\tilde{X}_j = 1-\tilde{z}_j$,
($1 \le j \le 3$), $T_s \equiv 1 - \zeta_s$ ($1 \le s \le 5$), the $\QK$ Whitney relations become
\begin{equation*} 
\sum_{i+j=\ell} e_i(z) e_j(\tilde{z}) \equiv e_\ell(\zeta) - 
\frac{q}{1-q} \left( \delta_{\ell,4}\Delta_1 + \delta_{\ell,5} (\Delta_1 - \Delta_2) \right) \/,
\end{equation*}
for $1 \le \ell \le 5$, with 
\[\begin{split} \Delta_1 = e_1(1-z_1, 1-z_2)(1-\tilde{z}_1)(1-\tilde{z}_2)(1-\tilde{z}_3)\/; \\
\Delta_2 = e_2(1-z_1, 1-z_2)(1-\tilde{z}_1)(1-\tilde{z}_2)(1-\tilde{z}_3) \/. \end{split}\]
One may solve for $e_i(\tilde{z})$ in terms of $e_i(z)$ to obtain:
\[ \begin{split} e_1(\tilde{z}) &\equiv e_1(\zeta) - h_1(z) = - h_1'(z,\zeta)\/; \\
e_2(\tilde{z}) & \equiv e_2(\zeta) - h_1(z) e_1(\zeta) + h_2(z) = - h_2'(z,\zeta) \/;\\
e_3(\tilde{z})& \equiv e_3(\zeta) - h_1(z) e_2(\zeta)+ h_2(z) e_1(\zeta) - h_3(z) = - h_3'(z,\zeta) \/. \end{split}
\]
This allows the elimination of the variables $\tilde{z}$. The remaining two relations,
and the Cauchy formula from \Cref{lemma:cauchy-h}
give the generators for the ideal $\widetilde{I}$:
\[ \begin{split} e_4(\zeta) - h_4'(z,\zeta) & = -e_1(z) h_3'(z,\zeta) + e_2(z) h_2'(z,\zeta) \\
& \equiv e_4(\zeta) -\frac{q}{1-q}(1+h_1(z)+h_2(z))(2-e_1(z)) \/. \end{split} \]
\[ \begin{split} e_5(\zeta) + h_5'(z,\zeta) - e_1(z) h_4'(z,\zeta) & = - e_2(z) h_3'(z,\zeta) \\
& \equiv e_5(\zeta) -\frac{q}{1-q}(1+h_1(z)+h_2(z))(1-e_2(z)) \/. \end{split} \]

We now turn to the Coulomb branch relations for $\QK_T(\Gr(2;5))$. The Vieta relations give
\begin{equation*}
\label{eq:g25physring}
	\sum_{i+j=\ell }e_i(z) e_j(\hat{z}) \equiv \hat{g}_\ell(z,\zeta,q) \quad \/, 
\end{equation*}
for $ 1 \le \ell \le 5$, where the polynomials $\hat{g}_{\ell}(z,\zeta,q)$ are given by
\begin{align*}
	\hat{g}_{\ell} (z,\zeta) &\:=\: e_\ell(\zeta) + e_{\ell-1}(\zeta) e_2(z),\quad \ell = 1, 2, 3, \\
	\hat{g}_4(z,\zeta) &\:=\: e_4(\zeta) + e_3(\zeta) e_2 (z) - e_5(\zeta) - q \sum_{s=0}^5 (-1)^s e_s(\zeta), \\
	\hat{g}_5(z,\zeta) &\:=\: e_5(\zeta) + e_2(z)e_4(\zeta)  - e_1(z) e_5(\zeta) - q \sum_{s=0}^5 (-1)^s e_s(\zeta).
\end{align*}
One may solve for $e_i(\tilde{z})$ in terms of $e_i(z)$ to obtain:
\[ \begin{split} e_1(\hat{z}) &= e_1(\zeta) - G_1(z) = - G_1'(z,\zeta)\/; \\
e_2(\hat{z}) &= e_2(\zeta) - G_1(z) e_1(\zeta) + G_2(z) = G_2'(z,\zeta) \/;\\
e_3(\hat{z})& = e_3(\zeta) - G_1(z) e_2(\zeta)+ G_2(z) e_1(\zeta) - G_3(z) = - G_3'(z,\zeta) \/. 
\end{split} \]
Here $G_i(z)$ are the Grothendieck polynomials, given by 
\[ \begin{split} G_1(z) & = h_1 - e_2= z_1 + z_2 - z_1 z_2 \/; \\ 
G_2(z) & = h_2 - s_{2,1}= z_1^2+z_1 z_2+z_2^2-z_1^2 z_2-z_1 z_2^2 \/; \\
G_3(z) & = h_3 - s_{3,1} = z_1^3+z_1^2 z_2+z_1 z_2^2+z_2^3-z_1^3 z_2-z_1^2 z_2^2-z_1 z_2^3 \/. 
\end{split}
\] 
After eliminating the variables $\hat{z}$, the remaining two relations,
and the Cauchy formula from \Cref{lemma:cauchy-h}(b)
give the generators for the ideal $\widehat{I}$ of $\widehat{\QK}_T(\Gr(2;5))$: 
\[ \begin{split} \hat{g}_4(z,\zeta,q) & \equiv - G_4'(z,\zeta) + c_{\ge 4}'=
 - G_4'(z,\zeta)+ e_4(\zeta) +e_3(\zeta) e_2(z) \\
\hat{g}_5(z,\zeta,q) & \equiv -e_1(z) G_4'(z,\zeta) + G_5'(z,\zeta) + c_{\ge 5}'
\\ & =-e_1(z) G_4'(z,\zeta) + G_5'(z,\zeta) + e_5(\zeta) + e_4(\zeta) e_2(z) 
\/, \end{split} \]
with the equivariant Grothendieck polynomials defined by 
\[ \begin{split} G_4'(z,\zeta) & = G_4(z) - G_3(z) e_1(\zeta) + G_2(z) e_2(\zeta) - G_1(z)e_3(\zeta) + e_4(\zeta) \/; \\ G_5'(z,\zeta) & = G_5(z) - G_4(z) e_1(\zeta) + G_3(z) e_2(\zeta) - G_2(z)e_3(\zeta) + G_1(z) e_4(\zeta) - e_5(\zeta) \/, \end{split}\] 
and $G_4(z), G_5(z)$ are given by \eqref{E:Gj}.


\bibliographystyle{alpha}
\bibliography{biblio.bib}

\end{document}